\documentclass[11pt]{amsart}
\usepackage[latin1]{inputenc}
\usepackage{amsmath,amsthm,amsfonts,amssymb,amscd}
\usepackage{verbatim}
\usepackage{mathrsfs}
\usepackage{multicol}
\usepackage{tabularx}
\usepackage[usenames,dvipsnames]{color}
\usepackage{enumerate}
\usepackage{graphicx}
\usepackage{color,xcolor}
\usepackage{bbm}
\usepackage[all]{xy}
\usepackage{hyperref}

\theoremstyle{definition}
\newtheorem{thm}{Theorem}[subsection]
\newtheorem{prop}[thm]{Proposition}
\newtheorem{cor}[thm]{Corollary}
\newtheorem{con}[thm]{Conjecture}
\newtheorem{lem}[thm]{Lemma}

\newtheorem{rem}[thm]{Remark}

\numberwithin{equation}{subsection}

\def\cal#1{\text{$\mathcal{#1}$}}
\def\ord#1^#2{#1$^{\text{#2}}$}
\def\lie#1{\mathfrak{#1}}
\def\tlie#1{\tilde{\mathfrak{#1}}}
\def\hlie#1{\hat{\mathfrak{#1}}}

\def\uqr#1^#2{\text{$U_q^{#2}(\lie #1)$}}

\def\uqhr#1^#2{\text{$U_q^{#2}(\hlie #1)$}}
\def\us#1^#2{\text{$U_{\xi}^{#2}(\lie #1)$}}
\def\ush#1^#2{\text{$U_{\xi}^{#2}(\hlie #1)$}}
\def\dus#1^#2{\text{$\dot{U}_{\xi}^{#2}(\lie #1)$}}
\def\dush#1^#2{\text{$\dot{U}_{\xi}^{#2}(\hlie #1)$}}

\def\wtl{{\rm wt}_\ell}
\def\wt{{\rm wt}}
\def\het{{\rm ht}}
\def\supp{{\rm supp}}
\def\rupp{{\rm rsupp}}

\def\ch{{\rm ch}}
\def\qch{{\rm qch}}

\def\opl_#1^#2{\text{\scriptsize$\bigoplus\limits_{\text{\normalsize$#1$}}^{\text{\normalsize$#2$}}$}}
\def\otm_#1^#2{\text{\scriptsize$\bigotimes\limits_{\text{\footnotesize$#1$}}^{\text{\footnotesize$#2$}}$}}
\def\wcal#1{{\mbox{$\widetilde{\cal #1}$}}}

\def\bs#1{\boldsymbol{#1}}

\def\endd{\hfill$\diamond$}

\newcommand{\g}{\mathfrak{g}}

\newcommand{\n}{\mathfrak{n}}

\newcommand{\C}{\mathbb{C}}

\newcommand{\Z}{\mathbb{Z}}

\begin{document}
\title[Non-Minimality of Certain Irregular Coherent Preminimal Affinizations]{Non-Minimality of Certain Irregular\\ Coherent Preminimal Affinizations}
\author[Adriano Moura and Fernanda Pereira]{Adriano Moura and Fernanda Pereira}
\thanks{Part of this work was developed while the second author was a visiting Ph.D. student at Institut de Mathématiques de Jussieu working under the supervision of David Hernandez. She thanks David Hernandez for his guidance during that period and Université Paris 7 for the hospitality. She also thanks Fapesp (grant 2009/16309-5) for the financial support. The work of the first author was partially supported by CNPq grant 304477/2014-1 and Fapesp grant 2014/09310-5. He also thanks Matheus Brito for useful discussions. Finally, we thank the anonymous referee for several important corrections as well as for the present simplified proof of Proposition \ref{p:parts}. }

\address{Departamento de Matemática, Universidade Estadual de Campinas, Campinas - SP - Brazil, 13083-859.}
\email{aamoura@ime.unicamp.br}

\address{Departamento de Matemática, Divisão de Ciências Fundamentais, Instituto Tecnológico de Aeronáutica, São José dos Campos - SP - Brazil, 12228-900.}
\email{fpereira@ita.br}

\begin{abstract}
Let $\mathfrak g$ be a finite-dimensional simple Lie algebra of type $D$ or $E$ and $\lambda$ be a dominant integral weight whose support bounds the subdiagram of type $D_4$. We study certain quantum affinizations of  the simple $\lie g$-module of highest weight $\lambda$ which we term preminimal affinizations of order two (this is the maximal order for such $\lambda$). This class can be split in two:  the coherent and the incoherent affinizations. 
If $\lambda$ is regular, Chari and Pressley proved that the associated minimal affinizations belong to one of the three equivalent classes of coherent preminimal affinizations. In this paper we show that, if $\lambda$ is irregular, the coherent preminimal affinizations are not minimal under certain hypotheses. Since these hypotheses are always satisfied if $\mathfrak g$ is of type $D_4$, this completes the classification of minimal affinizations for type $D_4$ by giving a negative answer to a conjecture of Chari-Pressley stating that the coherent and the incoherent affinizations were equivalent in type $D_4$ (this corrects the opposite claim made by the first author in a previous publication).
\end{abstract}

\maketitle

\section{Introduction}

This is the second paper of a series based on a project aiming at describing the classification of the Drinfeld polynomials of the irregular minimal affinizations of type $D$. The theory of minimal affinizations, initiated in \cite{cha:minr2,cp:minsl}, is object of intensive study due to its rich structure and connections to other areas such as mathematical physics and combinatorics \cite{chhe:bey,chi,her:min,her:KRconj,hele:KRcluster,liqi,naoi:grlim,naoi:D}.
We refer to the first two paragraphs of the first paper of the series \cite{hmp:tpa} for an account of the status of this classification problem when this project started. The main result of the present paper (Theorem \ref{t:cohnotmin}) is one of the crucial steps towards the final classification: it will provide one of the tools we shall use to compare certain affinizations or to show that they are not comparable. Moreover, together with the results of \cite{cp:minirr},  Theorem \ref{t:cohnotmin} completes the classification in type $D_4$ by showing that the corresponding elements of two different families of affinizations described in \cite{cp:minirr} are not equivalent, contrary to what was conjectured there. In fact, these elements are comparable and one is strictly larger than the other. Hence, only one of them, which we term an incoherent preminimal affinization of order 2 here, is actually a minimal affinization. The corresponding coherent affinization (the one that follows the pattern of the minimal affinizations in the regular case) is actually not minimal.
The study made here goes beyond type $D$ and proves that, under certain conditions, the coherent affinizations are not minimal also in type $E$. Since the proof of this result is rather lengthy and there are several parts which are interesting in their own right,  we deem appropriate to write a paper focusing exclusively on this result. 

We now give a not so formal description of the main results of the paper. Let $U_q(\tlie g)$ denote the quantum affine algebra associated to a simply laced finite type Kac-Moody algebra $\lie g$. If $\bs{\omega}$ is a Drinfeld polynomial with classical weight $\lambda$, then the associated irreducible $U_q(\tlie g)$-module $V_q(\bs{\omega})$, when regarded as module for $U_q(\lie g)$, decomposes as a direct sum of the form
\begin{equation*}
  V_q(\bs{\omega}) \cong V_q(\lambda)\oplus \oplus_{\mu}^{} V_q(\mu)^{\oplus m_\mu}
\end{equation*}
where the sum is over all dominant integral weights $\mu$ such that $\mu<\lambda$ in the usual partial order on weights, $m_\mu$ is a nonnegative  integer, and $V_q(\mu)$ denotes the irreducible $U_q(\lie g)$-module of highest weight $\mu$. Any module satisfying this kind of decomposition is said to be an affinization of $V_q(\lambda)$ and isomorphism as $U_q(\lie g)$-modules define an equivalence relation on the class of affinizatins of $V_q(\lambda)$. Moreover, the partial order on weights induces a partial order on the set of equivalence classes of affinizations which obviously admit minimal elements, termed minimal affinizations \cite{cha:minr2}. If $\lie g$ is of type $A$, since quantum analogues of evaluation maps exist \cite{jim:qan}, $V_q(\bs{\omega})$ is minimal if and only if $m_\mu=0$ for all $\mu<\lambda$. In that case, the roots of each polynomial $\bs{\omega}_i(u)$ form what is called a $q$-string, where $i\in I$ and $I$ is an index set for the nodes of the Dynkin diagram of $\lie g$. Moreover, if we denote by $a_i$ the center of the $q$-string associated to node $i$ and let $i_1$ be the first node in $\supp(\lambda)$, the support of $\lambda$, there exists a strictly monotonic function $f$ defined on $\supp(\lambda)$ such that
\begin{equation*}
 a_i = a_{i_1}q^{f(i)}. 
\end{equation*}
If $f$ is increasing, we say $V_q(\bs{\omega})$ is an increasing minimal affinization. Otherwise, we say it is decreasing.  Although the increasing and decreasing minimal affinizations are equivalent, the understanding of the combinations of increasing and decreasing patterns for diagram subalgebras of type $A$ is a key point for describing the minimal affinizations outside type $A$.

More precisely, it was proved in \cite{cp:minsl} that, if $V_q(\bs{\omega})$ is a minimal affinization and $J\subseteq I$ corresponds to a connected subdiagram of type $A$ which remains connected after removing the trivalent node $i_*$, then the associated $J$-tuple of polynomials correspond to a minimal affinization of type $A$. We say that  $\bs{\omega}$ is preminimal if it satisfies this property. Thus, if $V_q(\bs{\omega})$ is minimal,  $\bs{\omega}$ is preminimal. If $i$ is an extremal node, let $I_i$ be the maximal subdiagram of type $A$ not containing $i$. We say that $\bs{\omega}$ is $i$-minimal if the associated $I_i$-tuple of polynomials corresponds to a minimal affinization of type $A$. The order of minimality of a preminimal $\bs{\omega}$ is defined as the cardinality of the set of extremal nodes $i$ such that $\bs{\omega}$ is $i$-minimal. Hence, the order can be $0,1,2,3$. It follows from the results of \cite{cp:minsl} that, if $\supp(\lambda)\subseteq I_i$ for some extremal node $i$, then $V_q(\bs\omega)$ is a minimal affinization if and only if $\bs{\omega}$ is premininal of order $3$. On the other hand, if $\supp(\lambda)$ bounds the diagram of type $D_4$, the order of any preminimal $\bs{\omega}$ is at most $2$ and, if $i_*\in\supp(\lambda)$, $V_q(\bs\omega)$ is a minimal affinization if and only if $\bs{\omega}$ is premininal of order $2$. Note that, in this case, there are three equivalence classes of minimal affinizations, one for each extremal node, the node $i$ for which $i$-minimality fails. 

The purpose of this paper is to describe a few properties of the preminimal affinizations of order two when $i_*\notin\supp(\lambda)$ and $\supp(\lambda)$ bounds the subdiagram of type $D_4$. For type $D_4$, it follows from \cite{cp:minirr} that, if $V_q(\bs{\omega})$ is a minimal affinization, then $\bs{\omega}$ has order $1$ or $2$. It was clear from \cite{cp:minirr} that not all Drinfeld polynomials of order $1$ correspond to minimal affinizations. However, the conjecture mentioned in the first paragraph can be rephrased as ``all preminimal Drinfeld polynomials of order $2$ correspond to minimal affinizations''. The preminimal Drinfeld polynomials of order $2$ can be encoded by the following pictures:

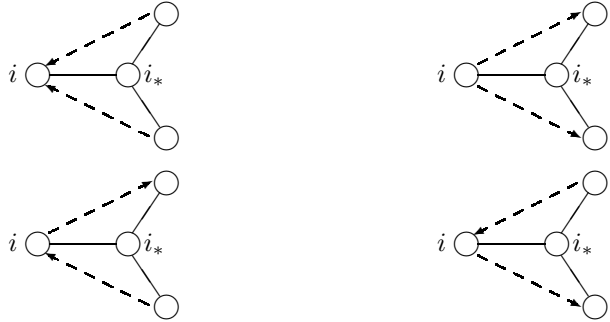
\setlength{\unitlength}{.3cm}
\begin{equation*}
  \begin{picture}(5,2)(17,0)
  \put(10,0.5){\circle{1}}\put(10.5,0.5){\line(1,0){3}}\put(14,0.5){\circle{1}}
  \put(14.2,1){\line(2,3){1.2}}\put(14.2,0){\line(2,-3){1.2}}\put(15.7,3.2){\circle{1}}\put(15.7,-2.28){\circle{1}}

  \multiput(10.5,1)(.8,.4){6}{\qbezier(0,0)(0,0)(.4,.2)}
  \multiput(10.5,0)(.8,-.4){6}{\qbezier(0,0)(0,0)(.4,-0.2)}

  \put(10.7,1.1){\vector(-2,-1){0.4}}  
  \put(10.7,-0.1){\vector(-2,1){0.4}}  
  
  \put(8.7,0.2){{\small${i}$}} \put(14.7,0.2){{\small${i_*}$}} 
  \end{picture}
  \begin{picture}(5,2)(3,0)
    \put(10,0.5){\circle{1}}\put(10.5,0.5){\line(1,0){3}}\put(14,0.5){\circle{1}}
    \put(14.2,1){\line(2,3){1.2}}\put(14.2,0){\line(2,-3){1.2}}\put(15.7,3.2){\circle{1}}\put(15.7,-2.28){\circle{1}}
    
    \multiput(10.5,1)(.8,.4){6}{\qbezier(0,0)(0,0)(.4,.2)}
    \multiput(10.5,0)(.8,-.4){6}{\qbezier(0,0)(0,0)(.4,-0.2)}

    \put(14.7,3.1){\vector(2,1){0.4}}
    \put(14.7,-2.1){\vector(2,-1){0.4}}
    
    \put(8.7,0.2){{\small${i}$}} \put(14.7,0.2){{\small${i_*}$}} 
  \end{picture}
\end{equation*}\vspace{5pt}

\vspace{5pt}
\begin{equation*}
\begin{picture}(5,2)(17,0)
\put(10,0.5){\circle{1}}\put(10.5,0.5){\line(1,0){3}}\put(14,0.5){\circle{1}}
\put(14.2,1){\line(2,3){1.2}}\put(14.2,0){\line(2,-3){1.2}}\put(15.7,3.2){\circle{1}}\put(15.7,-2.28){\circle{1}}

\multiput(10.5,1)(.8,.4){6}{\qbezier(0,0)(0,0)(.4,.2)}
\multiput(10.5,0)(.8,-.4){6}{\qbezier(0,0)(0,0)(.4,-0.2)}

\put(14.7,3.1){\vector(2,1){0.4}}
\put(10.7,-0.1){\vector(-2,1){0.4}}  

\put(8.7,0.2){{\small${i}$}} \put(14.7,0.2){{\small${i_*}$}} 
\end{picture}
\begin{picture}(5,2)(3,0)
\put(10,0.5){\circle{1}}\put(10.5,0.5){\line(1,0){3}}\put(14,0.5){\circle{1}}
\put(14.2,1){\line(2,3){1.2}}\put(14.2,0){\line(2,-3){1.2}}\put(15.7,3.2){\circle{1}}\put(15.7,-2.28){\circle{1}}

\multiput(10.5,1)(.8,.4){6}{\qbezier(0,0)(0,0)(.4,.2)}
\multiput(10.5,0)(.8,-.4){6}{\qbezier(0,0)(0,0)(.4,-0.2)}

\put(10.7,1.1){\vector(-2,-1){0.4}} 
\put(14.7,-2.1){\vector(2,-1){0.4}}

\put(8.7,0.2){{\small${i}$}} \put(14.7,0.2){{\small${i_*}$}} 
\end{picture}
\end{equation*}\vspace{5pt}

\noindent The arrows point in the direction that the function $f$ decreases and $i$ is the node for which $i$-minimality fails. Drinfeld polynomials satisfying either of the pictures in each line give rise to equivalent affinizations. Note that, if $i_*\in\supp(\lambda)$, Drinfeld polynomials satisfying the ones in the second line do not exist. We say that the ones satisfying the first line are coherent (because the arrows agree) and the ones satisfying the second line are incoherent.  The notion of coherent and incoherent preminimal Drinfeld polynomials of order $2$ can be extended for rank higher than $4$, including type $E$. Note that the incoherent ones do not exist when $\supp(\lambda)$ intersects more than once the connected subdiagram having $i$ and $i_*$ as extremal nodes. Otherwise, we conjecture that the coherent ones are not minimal affinizations and prove that this is indeed the case in type $D$ as well as under certain conditions on $\supp(\lambda)$ when $\lie g$ is of type $E$. In fact, the proof consists of showing that the coherent affinizations are strictly larger than their incoherent counterpart when the incoherent ones exist. 

Part of the proof of Theorem \ref{t:cohnotmin} is based on a result about the multiplicity of $V_q(\nu)$ as a summand of the coherent and incoherent affinizations where $\nu$ is a specific dominant weight.  The precise statement is in Proposition \ref{p:outer}, which can be proved in greater generality than Theorem \ref{t:cohnotmin}. The proof of Proposition \ref{p:outer} is a combination of arguments in the context of graded limits as well as in the context of qcharacters. Generators and relations for the graded limits of the coherent affinizations were described in \cite{mou} for type $D_4$. This result has been extended for type $D_n,n>4$, in \cite{naoi:D} in the case $i_*\in\supp(\lambda)$. However, as far as we can tell, the argument of \cite{naoi:D} also works when $i_*\notin\supp(\lambda)$. 
Inspired by Proposition \ref{p:outer}, we define in Section \ref{ss:grlim} a quotient of the ``coherent graded limits'' and prove that it projects onto the corresponding ``incoherent graded limits'' under the hypotheses of Theorem \ref{t:cohnotmin}. It is tempting to conjecture that these projections are actually isomorphisms (see Remark \ref{r:demazure} for further comments).

Beside Proposition \ref{p:outer}, by considering diagram subalgebras, the proof of Theorem \ref{t:cohnotmin} also relies on proving that certain tensor products of minimal affinizations in types $A$ and $D$ are irreducible. From type $A$ we need tensor products of a general minimal affinization with a Kirillov-Reshetikhin module supported on an extremal node. This  was exactly the topic of our first paper in this series, \cite{hmp:tpa}, where we described a necessary and sufficient condition for the irreducibility of such tensor products. In fact, this criterion for irreducibility is half of the main result of \cite{hmp:tpa}. The other half will be crucial in the proof of the final classification as it also provides a tool  to compare certain affinizations or to show that they are not comparable. From type $D$, we use a sufficient condition for the irreducibility of tensor products of two Kirillov-Reshetikhin modules associated to distinct extremal nodes proved in \cite{cha:braid}. For the   proof of the final classification of minimal affinizations in type $D$ we will need sharper results which will appear in \cite{hmp:tpd} (see also \cite{fer:thesis}). 

The paper is organized as follows. In Section \ref{ss:main}, we fix the basic notation and review the background needed to state the main results. Further technical background is reviewed in Section \ref{ss:tec}. In Section \ref{ss:parts}, we compute the dimension of certain weight spaces of certain $\lie g$-modules in terms of a ``modified'' Kostant partition function (see Proposition \ref{p:parts} and \eqref{e:dimVnuall}). Such dimensions play a crucial role in the proof of Proposition \ref{p:outer}. In Section \ref{ss:diag}, beside reviewing some known facts about diagram subalgebras, we prove Lemma \ref{l:many0} which is an important technical ingredient to be used in  the proof of Theorem \ref{t:cohnotmin}. The few facts about qcharacters that we need are reviewed in Section \ref{ss:qchar} and the aforementioned criteria for irreducibility of certain tensor products is reviewed in Section \ref{ss:tp}. Section \ref{ss:classlim} contains the needed technical background about graded limits. Section \ref{s:proof} is entirely dedicated to the proof of the main results: Theorem \ref{t:cohnotmin} and Proposition \ref{p:outer}. In Section \ref{ss:KR}, we deduce some facts about qcharacters and tensor products of Kirillov-Reshetikhin modules. Upper bounds for certain outer multiplicities are obtained in Section \ref{ss:upper} by studying graded limits. The main technical obstacles for proving Theorem \ref{t:cohnotmin} in greater generality when $\lie g$ is of type $E$ arise from Lemma \ref{l:remroots}. Although the extra hypotheses are necessary for the validity of that Lemma, they may not be necessary for the validity of Theorem \ref{t:cohnotmin}. However, the proof with the techniques employed here, would require a much more intricate analysis  (see the last paragraph of Section \ref{ss:cohinc} for more precise comments). The heart of the proof of Proposition \ref{p:outer} is in Sections \ref{ss:tpinc} and \ref{ss:tpcoh}, where we study irreducible factors of ``incoherent'' and ``coherent'' tensor products of Kirillov-Reshetikhin modules associated to extremal nodes of the Dynkin diagram. Theorem \ref{t:cohnotmin} is finally proved in Section \ref{ss:cohnotmin}.

\section{The Main Results}\label{ss:main}

Throughout the paper, let $\mathbb C$ and  $\mathbb Z$ denote the sets of complex numbers and integers, respectively. Let also $\mathbb Z_{\ge m} ,\mathbb Z_{< m}$, etc. denote the obvious subsets of $\mathbb Z$. Given a ring $\mathbb A$, the underlying multiplicative group of units is denoted by $\mathbb A^\times$. 
The symbol $\cong$ means ``isomorphic to''. We shall use the symbol $\diamond$ to mark the end of remarks, examples, and statements of results whose proofs are postponed. The symbol \qedsymbol\ will mark the end of proofs as well as of statements whose proofs are omitted. 

\subsection{Classical and Quantum Algebras}\label{ss:clalg}

Let $I$ be the set of vertices of a  finite-type simply laced indecomposable Dynkin diagram and let $\lie g$ be the associated simple Lie algebra over $\mathbb C$ with a fixed Cartan subalgebra $\lie h$. Fix a set of positive roots $R^+$ and let
$\lie g_{\pm \alpha},\alpha\in R^+$, and $\lie g=\lie n^-\oplus\lie h\oplus\lie n^+$ be the associated root spaces and triangular decomposition.
The simple roots will be denoted by $\alpha_i$, the fundamental weights by
$\omega_i$, $i\in I$, $Q,P,Q^+,P^+$ will denote the root and weight
lattices with corresponding positive cones, respectively.
Let also $h_\alpha\in\lie h$  be the co-root associated to $\alpha\in R^+$. If $\alpha=\alpha_i$ is simple, we often simplify notation and write
$h_i$. We denote by $x_{\alpha}^\pm$ any element spanning the root space $\lie g_{\pm\alpha}$. In particular, we shall write
\begin{equation}\label{e:nonormrv}
[x_\alpha^-,x_\beta^-]=x_{\alpha+\beta}^-.
\end{equation}
Let $C = (c_{ij})_{i,j\in I}$ be the Cartan matrix of $\lie g$, i.e., $c_{i,j}=\alpha_j(h_i)$.

By abuse of language, we will refer to any subset $J$ of $I$ as a subdiagram of the Dynkin diagram of $\lie g$. 
Given $J\subseteq I$, we let $\bar J$ be the minimal connected sudiagram of $I$ containing $J$ and let $\partial J$ be the subset of $J$ consisting of nodes connected to at most one other node of $J$ and 
\begin{equation*}
\mathring{J} = J\setminus\partial J.
\end{equation*}
For $i,j\in I$, set
\begin{equation*}
[i,j] = \overline{\{i,j\}}, \quad (i,j]= [i,j]\setminus\{i\}, \quad \quad [i,j)= [i,j]\setminus\{j\}.
\end{equation*}
Define also the distance between $i,j$ as
\begin{equation*}
d(i,j) = \#[i,j).
\end{equation*}
For a subdiagram $J\subseteq I$, we let $\lie g_J$ be the subalgebra of $\lie g$ generated by the corresponding simple root vectors, $\lie h_J=\lie h\cap\lie g_J$ and so on. Let also $Q_J$ be the subgroup of $Q$ generated by $\alpha_j, j\in J$, and
$R^+_J=R^+\cap Q_J$. Set
\begin{equation}\label{e:defvtheta}
\vartheta_J = \sum_{j\in J}\alpha_j
\end{equation}
which is an element of $R_J$ if $J$ is connected. 
When $J=I$ we may simply write $\vartheta$. 
Given $\lambda\in P$, let $\lambda_J$ denote the
restriction of $\lambda$ to $\lie h_J^*$ and let $\lambda^J\in P$ be
such that $\lambda^J(h_j)=\lambda(h_j)$ if $j\in J$ and
$\lambda^J(h_j)=0$ otherwise. The support of $\mu\in P$ is defined by
$${\rm supp}(\mu)=\{i\in I:\mu(h_i)\ne 0\}.$$ Given $\eta=\sum_{i\in I}s_i\alpha_i\in Q$, set
\begin{equation*}
\rupp(\eta) = \{i\in I:s_i\ne 0\}, \quad\het_i(\eta) = s_i,\ i\in I, \quad\text{and}\quad \het(\eta)=\sum_{i\in I} s_i.
\end{equation*}

For a Lie algebra $\lie a$ over $\mathbb C$, let $\tlie a=\lie a\otimes  \mathbb
C[t,t^{-1}]$ be its loop algebras and identify $\lie a$ with the subalgebra $\lie a\otimes 1$. 
Then, $\tlie g = \tlie n^-\oplus \tlie h\oplus \tlie n^+$ and $\tlie h$ is an abelian subalgebra. We denoted by $\lie a[t]$ the subalgebra determined by $\lie a\otimes\mathbb C[t]$. Let also $\lie a[t]_+=\lie a\otimes t\mathbb C[t]$. 
The elements $x_\alpha^\pm\otimes t^r, \alpha\in R^+, r\in\mathbb Z$, will be denoted by $x_{\alpha,r}^\pm$ and similarly we define $h_{\alpha,r}$. Given $a\in\mathbb C$, let $\tau_a$ be the Lie algebra automorphism of $\lie a[t]$ defined by 
\begin{equation}\label{e:grtaua}
\tau_a(x\otimes f(t))=x\otimes f(t-a) \quad\text{for every} \quad x\in\lie a,\ f(t)\in\mathbb C[t].
\end{equation}

Let $\mathbb F$ be an algebraic closure of $\mathbb C(q)$, the ring of rational functions on an indeterminate $q$, and let $U_q(\lie g)$ and $U_q(\tlie g)$ be the associated Drinfeld Jimbo quantum groups over $\mathbb F$. We use the notation as in \cite[Section 1.2]{mou}. In particular, the Drinfeld loop-like generators of $U_q(\tlie g)$ are denoted by $x_{i,r}^\pm, h_{i,s}, k_i^{\pm 1}, i\in I, r,s\in\mathbb Z, s\ne 0$. Also, $U_q(\lie g)$ is the subalgebra of $U_q(\tlie g)$ generated by $x_i^\pm = x_{i,0}^\pm, k_i^{\pm 1}, i\in I$, and the subalgebras $U_q(\lie n^\pm), U_q(\lie h), U_q(\tlie n^\pm), U_q(\tlie h), U_q(\lie g_J)$, $U_q(\tlie g_J)$, where $J\subseteq I$,  are defined in the expected way.

The $\ell$-weight lattice of $U_q(\tlie g)$ is the multiplicative group $\cal P_q$ of $n$-tuples of rational
functions $\bs\mu = (\bs\mu_i(u))_{i\in I}$ with values
in $\mathbb F$  such that $\bs\mu_i(0)=1$ for all $i\in I$. The elements of the submonoid $\cal P^+_q$ of
$\cal P_q$ consisting of $n$-tuples of polynomials will be refered to as dominant $\ell$-weights or Drinfeld polynomials. Given
$a\in\mathbb F^\times$ and $\mu\in P$, let $\bs{\omega}_{\mu,a}\in\mathcal P_q$ be the element whose $i$-th rational function is
\begin{equation*}
(1-au)^{\mu(h_i)}, \quad i\in I.
\end{equation*}
In the case that $\mu=\omega_i$ for some $i$, we often simplify notation and write $\bs\omega_{i,a}$. Note that $\mathcal P_q$ is the (multiplicative) free abelian group on the set $\{\bs{\omega}_{i,a}:i\in I,a\in\mathbb F^\times\}$, let $\mathcal P$ denote the subgroup generated by $\{\bs{\omega}_{i,a}:i\in I,a\in\mathbb C^\times\}$, and $\mathcal P^+=\mathcal P_q^+\cap\mathcal P$. If
\begin{equation}\label{e:appears}
\bs\mu = \prod_{(i,a)\in I\times\mathbb F^\times} \bs\omega_{i,a}^{p_{i,a}}
\end{equation}
we shall say that $\bs\omega_{i,a}$ (respectively,
$\bs\omega_{i,a}^{-1}$) appears in $\bs\mu$ if $p_{i,a}>0$
(respectively, $p_{i,a}<0$).
Let $\wt:\mathcal P_q \to P$ be the group homomorphism determined by setting $\wt(\bs\omega_{i,a})=\omega_i$. 
We have an injective map $\bs\Psi:\mathcal P_q\to (U_q(\tlie h))^*, \bs{\omega\mapsto \bs\Psi_{\bs{\omega}}}$, (see \cite[Section 1.8]{mou}) and, hence, we identify $\cal P_q$  with its image in $(U_q(\tlie h))^*$ under $\bs\Psi$. Similarly, there is an injective map $\mathcal P\to \tlie h^*$. 
Following \cite{cm:qblock}, given $i\in I, a\in\mathbb F^\times, m\in\mathbb Z_{\ge 0}$, define
\begin{equation*}
\bs\omega_{i,a,m} = \prod_{j=0}^{m-1} \bs\omega_{i,aq^{m-1-2j}} \qquad\text{and}\qquad
\bs\alpha_{i,a} = \bs\omega_{i,aq,2}\prod_{j\ne i}
\bs\omega_{j,aq,-c_{j,i}}^{-1}
\end{equation*}

For $\bs\omega\in\cal P_q$, let $\bs\omega_J$ be the associated
$J$-tuple of rational functions and let $\cal
P_J=\{\bs\omega_J:\bs\omega\in \cal P_q\}$. Similarly define $\cal
P_J^+$. Notice that $\bs\omega_J$ can be regarded as an element of
the $\ell$-weight lattice of $U_q(\tlie g_J)$. Let $\pi_J:\cal P_q\to
\cal P_J$ denote the map $\bs\omega\mapsto\bs\omega_J$. If $J=\{j\}$
is a singleton, we write $\pi_j$ instead of $\pi_J$. An
$\ell$-weight $\bs\omega$ is said to be $J$-dominant if
$\bs\omega_J\in \cal P_J^+$. Let also $\cal Q_J\subset\cal P_J$ (respectively, $\cal Q_J^+$) be
the subgroup (submonoid) generated by $\pi_J(\bs\alpha_{j,a}), j\in
J,a\in\mathbb C^\times$. When no confusion arises, we shall simply
write $\bs\alpha_{j,a}$ for its image in $\cal P_J$ under $\pi_J$.
Let
$$\iota_J:\Z[\cal Q_J]\rightarrow \Z[\cal Q_q],$$
be the ring homomorphism such that
$\iota_J(\bs\alpha_{j,a})=\bs\alpha_{j,a}$ for all $j\in J,
a\in\mathbb C^\times$. We shall often abuse of notation and identify
$\cal Q_J$ with its image under $\iota_J$. In particular, given
$\bs\mu\in\cal P_q$, we set
\begin{equation*}
\bs\mu\cal Q_J =\{ \bs\mu\bs\alpha: \bs\alpha\in\iota_J(\cal Q_J)\}.
\end{equation*}
It will also be useful to introduce the element $\bs\omega^J\in\cal P_q$ defined by
$$(\bs\omega^J)_j(u)=\bs\omega_j(u) \quad\text{if}\quad j\in J \quad\text{and}\quad (\bs\omega^J)_j(u)=1 \quad\text{otherwise.}$$

\subsection{Finite-Dimensional Representations}
We let $\mathcal C$ denote the category of finite-dimen\-sional $\lie g$-modules and denote by $V(\lambda)$ an irreducible $\lie g$-module of highest weight $\lambda\in P^+$. The character of a $\lie g$-module $V$ will be denoted by $\ch(V)$. We think of $\ch(V)$ as an element of the group ring $\mathbb Z[P]$. Let also $\cal C_q$ be the category of all finite-dimensional type 1 modules of $U_q(\lie g)$. Thus, a finite-dimensional $U_q(\lie g)$-module $V$ is in $\cal C_q$ if $V=\bigoplus_{\mu\in P}^{} V_\mu$ where
$$V_\mu=\{v\in V: k_iv=q^{\mu(h_i)}v \text{ for all } i\in I\}.$$
The character of $V$, also denote by $\ch(V)$, is defined in the obvious way.  
The following theorem summarizes the basic facts about $\cal C_q$.

\begin{thm}\label{t:ciuqg} Let $V$ be an object of $\cal C_q$. Then:
\begin{enumerate}[(a)]
\item $\dim V_\mu = \dim V_{w\mu}$ for all $w\in\cal W$.
\item $V$ is completely reducible.
\item For each $\lambda\in P^+$ the $U_q(\lie g)$-module $V_q(\lambda)$ generated by a vector $v$ satisfying
$$x_i^+v=0, \qquad k_iv=q^{\lambda(h_i)}v, \qquad (x_i^-)^{\lambda(h_i)+1}v=0,\quad\forall\ i\in I,$$
is irreducible and finite-dimensional. If $V\in\cal C_q$ is
irreducible, then $V$ is isomorphic to $V_q(\lambda)$  for some
$\lambda\in P^+$.
\item For all $\lambda\in P^+$,  $\ch(V_q(\lambda)) = \ch(V(\lambda))$.
\hfill\qedsymbol
\end{enumerate}
\end{thm}

If $J\subseteq I$, we shall denote by $V_q(\lambda_J)$ the simple
$U_q(\lie g_J)$-module of highest weight $\lambda_J$. Since $\cal
C_q$ is semisimple, it is easy to see that, if $\lambda\in P^+$ and
$v\in V_q(\lambda)_\lambda$ is nonzero, then $U_q(\lie g_J)v\cong
V_q(\lambda_J)$.

Let $\wcal C_q$ the category of all finite-dimensional $\ell$-weight modules of $U_q(\tlie g)$.
Thus, a finite-dimensional $U_q(\tlie g)$-module $V$ is in $\wcal C_q$ if
$$V=\bigoplus_{\bs\omega\in\mathcal P_q}^{} V_{\bs\omega}$$
where
$$v\in V_{\bs\omega} \quad\Leftrightarrow\quad \exists\ k\gg 0 \quad\text{s.t.}\quad (\eta-\bs\Psi_{\bs\omega}(\eta))^kv=0 \quad\text{for all}\quad \eta\in U_q(\tlie h).$$
$V_{\bs\omega}$ is called the $\ell$-weight space of $V$ associated to $\bs\omega$. A nonzero vector $v\in V_{\bs\omega}$ is
said to be a highest-$\ell$-weight vector if
$$\eta v=\bs\Psi_{\bs\omega}(\eta)v \quad\text{for every}\quad \eta\in U_q(\tlie h)
\quad\text{and}\quad x_{i,r}^+v=0 \quad\text{for all}\quad i\in I, r\in\mathbb Z.$$
$V$ is said to be a highest-$\ell$-weight module if it is generated by a
highest-$\ell$-weight vector.
Note that if $V\in\wcal C_q$, then $V\in\cal C_q$ and
\begin{equation*}
V_\lambda = \bigoplus_{\bs\omega:\wt(\bs\omega)=\lambda}^{}
V_{\bs\omega}.
\end{equation*}
Given $\bs{\omega}\in\mathcal P_q^+$, the local Weyl module $W_q(\bs{\omega})$ is the $U_q(\tlie g)$-module generated by a vector $w$ satisfying the defining relations
\begin{equation*}
x_{i,r}^+w=0, \quad xw = \bs{\Psi}_{\bs{\omega}}(x)w, \quad (x_i^-)^{\wt(\bs{\omega})+1}w=0
\end{equation*}
for all $i\in I, r\in\mathbb Z, x\in U_q(\tlie h)$. It was proved in \cite{cp:weyl} that $W_q(\bs{\omega})\in\widetilde{\mathcal C}_q$ for every $\bs{\omega}\in\mathcal P_q^+$ and every finite-dimensional highest-$\ell$-weight module of highest $\ell$-weight $\bs{\omega}$ is a quotient of $W_q(\bs{\omega})$. Standard arguments show that $W_q(\bs{\omega})$ has a unique irreducible quotient, denoted by $V_q(\bs{\omega})$. In particular, we get the following classification of the simple objects of $\wcal C_q$.

\begin{thm}\label{t:weyl}
If $V$ is a simple object of $\wcal C_q$, then $V$ is isomorphic to $V_q(\bs{\omega})$ for some $\bs\omega\in\mathcal P^+_q$.\hfill\qedsymbol
\end{thm}

Analogous results hold for the category $\mathcal C$ of finite-dimensional $\tlie g$-modules. In particular, given $\bs{\omega}\in\mathcal P^+$, we let $W(\bs{\omega})$ and $V(\bs{\omega})$ denote the corresponding local Weyl module and irreducible module, respectively. The following is a corollary of the proof that $W(\bs{\omega})$ is finite-dimensional. 

\begin{prop}\label{p:genbycurrent}
	Suppose $V$ is a highest-$\ell$-weight module for $\tlie g$ and let $v$ be a highest-$\ell$-weight vector. Then, $V=U(\lie n^-[t])v$.\hfill\qedsymbol
\end{prop}

\subsection{Minimal affinizations}\label{ss:min}
Since $\mathcal C_q$ is semisimple, for any object $V\in\widetilde{\mathcal C}_q$ we have an isomorphism of $U_q(\lie g)$-modules
\begin{equation*}
V\cong \opl_{\mu\in P^+}^{} V_q(\mu)^{\oplus m_\mu(V)}
\end{equation*}
for some $m_\mu(V)\in\mathbb Z_{\ge 0}$. We shall refer to the number $m_\mu(V)$ as the  multiplicity of $V_q(\mu)$ in $V$. 

Given $\lambda\in P^+$, $V\in\widetilde{\mathcal C}_q$ is said to be an affinization of $V_q(\lambda)$ if 
\begin{equation}
m_\lambda(V)=1  \qquad\text{and}\qquad m_\mu(V)\ne 0\quad\Rightarrow\quad \mu\le\lambda.
\end{equation}
Two affinizations of $V_q(\lambda)$ are said to be equivalent if they are isomorphic as $U_q(\lie g)$-modules.  The partial order on $P^+$ induces a natural partial order on the
set of  (equivalence classes of) affinizations of $V_q(\lambda)$.
Namely, if $V$ and $W$ are affinizations of $V_q(\lambda)$, say that
$V\le W$ if, for each $\mu\in P^+$, one of the following conditions holds:
\begin{enumerate}
\item $m_\mu(V)\le m_\mu(W)$;
\item if $m_\mu(V)>m_\mu(W)$, there exists $\nu>\mu$ such that $m_\nu(V)<m_\nu(W)$.
\end{enumerate}
A minimal element of this partial order is said to be a minimal affinization \cite{cha:minr2}. Clearly, a minimal affinization of $V_q(\lambda)$ must be irreducible as a $U_q(\tlie g)$-module and, hence, is of the form $V_q(\bs\omega)$ for some $\bs\omega\in\cal P^+_q$ such that $\wt(\bs\omega)=\lambda$. More generally, any quotient of $W_q(\bs{\omega})$ is an affinization of $V_q(\lambda)$ provided $\wt(\bs\omega)=\lambda$.

Suppose $\lie g$ is of type $A_n$. It follows from \cite{jim:qan} that $V_q(\lambda)$ has a unique equivalence class of minimal affinizations and $V_q(\bs{\omega})$ is a minimal affinization if, and only if, it is an irreducible $U_q(\g)$-module. To describe the elements of $\mathcal P_q^+$ with this property,  identify $I$ with $\{1,2,\dots,n\}$ in such a way that $c_{i,i+1}=-1$ for all $1\le i<n$ and $\omega_1$ is the highest weight of the standard representation of $\lie g$. Given $i,j\in I, i\le j$,  set
\begin{equation*}
_i|\lambda|_j = \sum_{k=i}^j \lambda(h_k).
\end{equation*}
If $i=1$, we may write $|\lambda|_j$ instead of $_1|\lambda|_j$ and similarly if $j=n$. For $i>j$, we set
$_i|\lambda|_j=0$ and, for $i\le j$, define
\begin{equation*}
p_{i,j}(\lambda) =  p_{j,i}(\lambda)= \ _{i+1}|\lambda|_j +\ _i|\lambda|_{j-1} + j-i.
\end{equation*}
In particular, $p_{i,j}(\lambda)=0$ if $i=j$ and
\begin{equation*}
p_{i,j}(\lambda) = \lambda(h_i)+\lambda(h_j)+ 2\
_{i+1}|\lambda|_{j-1}+ j-i \qquad\text{if}\quad i<j.
\end{equation*}

\begin{thm}\cite{cp:small}\label{teo tipo A}
$V_q(\bs\omega)$ is a minimal affinization of $V_q(\lambda)$ if and only if there exist $a_i\in\mathbb F^\times, i\in I$, and
$\epsilon=\pm 1$ such that
\begin{equation}\label{eq cond min A}
\bs\omega=\prod_{i\in I} \bs\omega_{i,a_i,\lambda(h_i)}
\qquad\text{with}\qquad \frac{a_i}{a_j}= q^{\epsilon
p_{i,j}(\lambda)} \qquad\text{for all}\quad i<j.
\end{equation}\hfill\qedsymbol
\end{thm}

Notice that \eqref{eq cond min A} is equivalent to saying that there
exist $a\in\mathbb F^\times$ and $\epsilon=\pm 1$ such that
\begin{equation}\label{eq cond min A'}
\bs\omega=\prod_{i\in I} \bs\omega_{i,a_i,\lambda(h_i)}
\qquad\text{with}\qquad a_i=aq^{\epsilon p_{i,n}(\lambda)}
\qquad\text{for all}\quad i\in I.
\end{equation}
If $\#\ \supp(\lambda)>1$, the pair $(a,\epsilon)$ in
\eqref{eq cond min A'} is unique. In that case, if $\bs\omega$
satisfies \eqref{eq cond min A} with $\epsilon=1$, we say that
$V_q(\bs\omega)$ is a decreasing minimal affinization.
Otherwise, we say $V_q(\bs\omega)$ is an increasing minimal affinization. 
If $\#\ \supp(\lambda)= 1$, $\bs\omega$ can be represented in the form \eqref{eq cond min A'} by two choices of pairs $(a,\epsilon)$, one for each value of $\epsilon$. We do not fix a preferred presentation in that case. We consider $\bs{\omega}$ to be simultaneously increasing and decreasing if $\#\ \supp(\lambda)\le 1$.

Assume now that $\lie g$ is of type $D$ or $E$, let $i_*$ be the trivalent node, $\bs{\omega}\in\mathcal P_q^+$ and $\lambda=\wt(\bs{\omega})$. We will say that $\bs{\omega}$ is preminimal if $V_q(\bs{\omega}_J)$ is a minimal affinization for any connected subdiagram $J$ of type $A$ such that $J\setminus\{i_*\}$ is connected. It was proved in \cite[Proposition 4.2]{cp:minsl} that, if $V_q(\bs{\omega})$ is a minimal affinization, then $\bs{\omega}$ is preminimal. Henceforth, we assume  $\bs{\omega}$ is preminimal. It will be proved in Lemma \ref{l:many0} that
\begin{equation}\label{e:Jmu}
m_\mu(V_q(\bs{\omega}))>0 \quad\Rightarrow\quad J_\mu \ \text{is connected, where}  \quad J_\mu = \rupp(\lambda-\mu).
\end{equation}

Given $i\in\partial I$, let
\begin{equation*}
I_i = \overline{\partial I\setminus\{i\}}
\end{equation*}
Thus, $I_i$ is the maximal connected subdiagram of type $A$ which does not contain $i$, $i_*\in I_i$, and $I_i\setminus\{i_*\}$ is disconnected.  We will say that $\bs{\omega}$ is $i$-minimal if $V_q(\bs{\omega}_{I_i})$  is a minimal affinization. Define the minimality order of  $\bs{\omega}$  as 
\begin{equation*}
{\rm mo}(\bs{\omega}) = \#\{i\in \partial I: V_q(\bs{\omega}) \text{ is  $i$-minimal}\}.
\end{equation*}
If $\bs{\omega}$ is preminimal of minimality order $k$, we shall simply say  $\bs{\omega}$ is preminimal of order $k$. 
The minimality order of $V_q(\bs{\omega})$ is set to be ${\rm mo}(\bs{\omega})$. Note that, if ${\rm mo}(\bs{\omega}) =3$, then $V_q(\bs{\omega}_J)$ is a minimal affinization for every connected sudiagram $J$ of type $A$. One easily checks using Theorem \ref{teo tipo A} that
\begin{equation}
{\rm mo}(\bs{\omega}) = 3 \quad\Rightarrow\quad \overline{\supp}(\lambda) \text{ is of type }A \quad\Rightarrow\quad {\rm mo}(\bs{\omega})\ge 2.
\end{equation}
The next Theorem was proved in \cite{cp:minsl,cp:minirr}.

\begin{thm}\label{t:regmin}\hfill \vspace{-5pt}
	
	\begin{enumerate}[(a)]
		\item If $\overline{\supp}(\lambda)$ is of type $A$, then $V_q(\bs{\omega})$ is a minimal affinization if and only if $\bs{\omega}$ is is preminimal of order $3$. In particular, $V_q(\lambda)$ has a unique equivalence class of minimal affinizations.
		\item If $\overline{\supp}(\lambda)$ is not of type $A$ and $\lambda(h_{i_*})\ne 0$, then $V_q(\bs{\omega})$ is a minimal affinization if and only if $\bs{\omega}$ is preminimal of order $2$. In particular, $V_q(\lambda)$ has $3$ equivalence classes of minimal affinizations.
		\item If $\lie g$ is of type $D_4$, $\overline{\supp}(\lambda)$ is not of type $A$, $\lambda(h_{i_*})= 0$, and $V_q(\bs{\omega})$ is a minimal affinization, the order of $\bs{\omega}$ is either $2$ or $1$. Moreover, the number of equivalence classes of minimal affinizations of order $1$ grows unboundedly with $\lambda$.\hfill\qedsymbol
	\end{enumerate}
\end{thm}

Outside type $A$, a minimal affinization is typically not irreducible as a $U_q(\lie g)$-module even under the assumption of part (a) of Theorem \ref{t:regmin}. If $\lambda$ satisfies the hypothesis of either part (a) or (b) of Theorem \ref{t:regmin} it is said to be regular. Otherwise it is said to be irregular. If $\lie g$ is of type $D$ and $\lambda$ is regular, the character of the minimal affinizations were computed in \cite{naoi:grlim,naoi:D} in terms of Demazure operators.

\subsection{Coherent and Incoherent Affinizations}\label{ss:cohinc} Assume $\lie g$ is of type $D$ or $E$. For a connected subdiagram $J\subseteq I$ of type $A$, we shall say that a total ordering $<$ on $J$ is monotonic if, for all $i,k\in J$, we have
\begin{equation*}
 c_{i,k}=-1 \text{ and } i<k \qquad\Rightarrow\qquad \{j\in J:i<j<k\}=\emptyset.
\end{equation*}
For each such subdiagram there are exactly two choices of monotonic orderings and the maximum and minimum of a monotonic ordering belong to $\partial J$.
Each monotonic ordering on $J$ with the corresponding order-preserving identification of $J$ with $\{1,\dots,n\}, n=\#J$, induces an isomorphism of algebras
\begin{equation*}
 U_q(\tlie g_J)\cong U_q(\tlie{sl}_{n+1}).
\end{equation*}
Then, given $\bs{\omega}\in\mathcal P_q^+$ such that $V_q(\bs{\omega}_J)$ is a minimal affinization, 
we shall say that $\bs{\omega}_J$ is increasing or decreasing, with respect to the given ordering, if $V_q(\bs{\omega}_J)$ is an increasing or a decreasing minimal affinization for $U_q(\tlie{sl}_{n+1})$ after pulling back the action by the above isomorphism. If $J'$ is another connected subdiagram of type $A$, we shall say that a choice of monotonic orderings on $J$ and $J'$ is coherent if they coincide on $J\cap J'$.  In that case, we shall say that $J$ and $J'$ are coherently ordered. Evidently, given any two intersecting such diagrams, there exists at least one choice of coherent orderings. The following lemma is also easily established.

\begin{lem}
	Let $J,J'$ be coherently ordered connected subdiagrams of type $A$. Suppose $\bs{\omega}\in\mathcal P_q^+$ is such that both $V_q(\bs{\omega}_J)$ and $V_q(\bs{\omega}_{J'})$ are minimal affinizations and that
	\begin{equation*}
	  \#(\supp(\wt(\bs{\omega}))\cap J\cap J')>1. 
	\end{equation*}
	Then,  $\bs{\omega}_J$ is increasing if, and only if, $\bs{\omega}_{J'}$ is increasing.\hfill\qed
\end{lem}

Note that the assumption on the cardinality is essential in the above lemma. 

Suppose $\bs{\omega}\in\mathcal P_q^+$ is preminimal of order $2$, let $k\in\partial I$ be the node such that $\bs\omega$ is not  $k$-minimal, and choose coherent monotonic orderings on $I_l, l\in\partial I\setminus\{k\}$.  
We shall say that $\bs{\omega}$ is coherent if $V_q(\bs{\omega}_{I_l}), l\ne k$, are either both increasing or both decreasing minimal affinizations. Otherwise, we say that $\bs{\omega}$ is incoherent. Note that the property of being coherent is intrinsic to $\bs{\omega}$, i.e., it does not depend on the choice of coherent monotonic orderings on $I_l, l\in\partial I_k$. 
Moreover, it follows from the previous lemma that
\begin{equation}\label{e:noincreg}
  \bs{\omega} \text{ incoherent} \qquad\Rightarrow\qquad \#(\supp(\bs{\omega})\cap [k,i_*])\le 1.
\end{equation}
 We can graphically represent coherent Drifeld polynomials by the pictures
 
 \setlength{\unitlength}{.3cm}
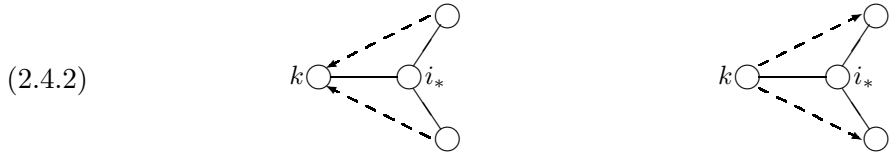
 \begin{equation}
 \begin{picture}(5,2)(17,0)
 \put(10,0.5){\circle{1}}\put(10.5,0.5){\line(1,0){3}}\put(14,0.5){\circle{1}}
 \put(14.2,1){\line(2,3){1.2}}\put(14.2,0){\line(2,-3){1.2}}\put(15.7,3.2){\circle{1}}\put(15.7,-2.28){\circle{1}}
 
 \multiput(10.5,1)(.8,.4){6}{\qbezier(0,0)(0,0)(.4,.2)}
 \multiput(10.5,0)(.8,-.4){6}{\qbezier(0,0)(0,0)(.4,-0.2)}

 \put(10.7,1.1){\vector(-2,-1){0.4}} 
 \put(10.7,-0.1){\vector(-2,1){0.4}}  
 
 \put(8.7,0.2){{\small${k}$}} \put(14.7,0.2){{\small${i_*}$}} 
 \end{picture}
 \begin{picture}(5,2)(3,0)
 \put(10,0.5){\circle{1}}\put(10.5,0.5){\line(1,0){3}}\put(14,0.5){\circle{1}}
 \put(14.2,1){\line(2,3){1.2}}\put(14.2,0){\line(2,-3){1.2}}\put(15.7,3.2){\circle{1}}\put(15.7,-2.28){\circle{1}}

 \multiput(10.5,1)(.8,.4){6}{\qbezier(0,0)(0,0)(.4,.2)}
 \multiput(10.5,0)(.8,-.4){6}{\qbezier(0,0)(0,0)(.4,-0.2)}
 
 \put(14.7,3.1){\vector(2,1){0.4}}
 
 \put(14.7,-2.1){\vector(2,-1){0.4}}
 
 \put(8.7,0.2){{\small${k}$}} \put(14.7,0.2){{\small${i_*}$}} 
 \end{picture}
 \end{equation}\vspace{5pt}

\noindent where the first means it is decreasing towards $k$ and the second it is decreasing away from $k$. Similarly, incoherent ones can be represented by 

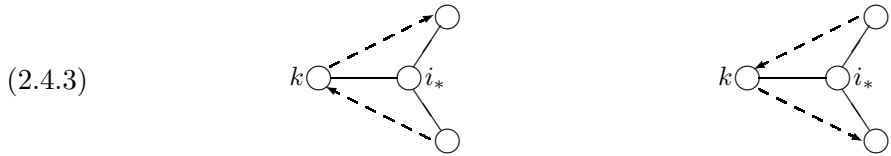
\begin{equation}\label{e:incpic}
\begin{picture}(5,2)(17,0)
\put(10,0.5){\circle{1}}\put(10.5,0.5){\line(1,0){3}}\put(14,0.5){\circle{1}}
\put(14.2,1){\line(2,3){1.2}}\put(14.2,0){\line(2,-3){1.2}}\put(15.7,3.2){\circle{1}}\put(15.7,-2.28){\circle{1}}

\multiput(10.5,1)(.8,.4){6}{\qbezier(0,0)(0,0)(.4,.2)}
\multiput(10.5,0)(.8,-.4){6}{\qbezier(0,0)(0,0)(.4,-0.2)}

\put(14.7,3.1){\vector(2,1){0.4}}
\put(10.7,-0.1){\vector(-2,1){0.4}}  

\put(8.7,0.2){{\small${k}$}} \put(14.7,0.2){{\small${i_*}$}} 

\end{picture}
\begin{picture}(5,2)(3,0)
\put(10,0.5){\circle{1}}\put(10.5,0.5){\line(1,0){3}}\put(14,0.5){\circle{1}}
\put(14.2,1){\line(2,3){1.2}}\put(14.2,0){\line(2,-3){1.2}}\put(15.7,3.2){\circle{1}}\put(15.7,-2.28){\circle{1}}

\multiput(10.5,1)(.8,.4){6}{\qbezier(0,0)(0,0)(.4,.2)}
\multiput(10.5,0)(.8,-.4){6}{\qbezier(0,0)(0,0)(.4,-0.2)}

\put(10.7,1.1){\vector(-2,-1){0.4}}  
\put(14.7,-2.1){\vector(2,-1){0.4}}

\put(8.7,0.2){{\small${k}$}} \put(14.7,0.2){{\small${i_*}$}} 
\end{picture}
\end{equation}\vspace{5pt}

\noindent These pictures are inspired by those in the main theorem of \cite{cp:minirr}. More involved pictures appear in the main theorem of \cite{fer:thesis}.

\begin{con}\label{cj:cohnotmin}
	Let $\bs{\omega}\in\mathcal P_q^+$ be preminimal of order $2$ and let $k\in\partial I$ be the node such that $\bs{\omega}$ is not $k$-minimal. If 	$\#(\supp(\wt(\bs{\omega}))\cap [i_*,k])\le 1$ and $\bs{\omega}$ is coherent, $V_q(\bs{\omega})$ is not a minimal affinization. \endd
\end{con}

\begin{rem}
Note that, if  $\supp(\wt(\bs{\omega}))\subseteq I_k$, this conjecture follows from part (a) of Theorem \ref{t:regmin}, since the minimal affinizations have minimality order $3$. In that case, note that, if $\#(\supp(\wt(\bs{\omega}))\cap I_l)\ge 2$ for $l\in\partial I_k$, a graphic representation of the Drinfeld polynomial of the minimal affinizations follows the picture \eqref{e:incpic}. In other words, we can informally say that the minimal affinizations are incoherent, even though the notion is not defined when the minimality order is $3$. 

It follows from the above paragraph that it remains to prove the conjecture in the case $\#(\supp(\wt(\bs{\omega}))\cap [i_*,k])= 1$ and $i_*\notin \supp(\wt(\bs{\omega}))$. If $\#(\supp(\wt(\bs{\omega}))\cap [i_*,k])>1$, the conclusion of the conjecture is false by part (b) of Theorem \ref{t:regmin}. In fact, in the context of Theorem \ref{t:regmin}(b), $V_q(\bs{\omega})$ is a minimal affinization if and only if $\bs{\omega}$ is coherent. We shall see in a forthcoming publication that the conclusion of the conjecture remains false if  $\#(\supp(\wt(\bs{\omega}))\cap [i_*,k])>1$ and $i_*\notin \supp(\wt(\bs{\omega}))$ (see also \cite{fer:thesis}). Note that this situation is realizable only if $\#I>4$.\endd
\end{rem}

We will prove that the conclusion of Conjecture \ref{cj:cohnotmin} holds under certain extra hypotheses in the case $\lie g$ is of type $E$. To state them, we introduce the following notation. 
Given $\lambda\in P^+$ and $i\in\partial I$, if $\supp(\lambda)\cap (i_*,i]\ne\emptyset$, let $i_\lambda \in \supp(\lambda)\cap (i_*,i]$ be the element which is closest to $i_*$. Otherwise, set $i_\lambda=i$. Set also
\begin{equation*}
I^\lambda=\overline{\{i_\lambda:i\in \partial I\}}\qquad \text{and}\qquad I_i^\lambda = I_i\cap I^\lambda \quad\text{for}\quad i\in\partial I.
\end{equation*}
Note  $i_\lambda\in I_j^\lambda$ if, and only if, $j\ne i$. 
The main result of this paper is:

\begin{thm}\label{t:cohnotmin}
		Let $\bs{\omega}\in\mathcal P_q^+$ be preminimal of order $2$ and let $k\in\partial I$ be the node such that $\bs{\omega}$ is not $k$-minimal. Set $\lambda=\wt(\bs{\omega})$ and assume $\#(\supp(\lambda)\cap [i_*,k])= 1, i_*\notin\supp(\lambda)$, and $\bs{\omega}$ is coherent. Then, $V_q(\bs{\omega})$ is not a minimal affinization provided either one of the following hypothesis holds:
	\begin{enumerate}[(i)]
		\item $\lie g$ is of type $D$,
		\item $I^\lambda$ is of type $D_4$ and $d(k,i_*)>1$,
		\item $\lie g$ is of type $E_6$ and $\supp(\lambda)=\partial I$. 
	\end{enumerate}
	 More precisely, $V_q(\bs{\omega})>V_q(\bs{\varpi})$ for any $\bs{\varpi}\in\mathcal P_q^+$ which is preminimal of order $2$, not  $k$-minimal, incoherent, and such that $\wt(\bs{\varpi})=\lambda$.\endd
\end{thm}

\begin{rem}
	This completes the classification of minimal affinizations for $\lie g$ of type $D_4$. Namely, it was proved in \cite{cp:minirr} that, if $V_q(\bs{\omega})$ is an irregular minimal affinization, then it must belong to 3 explicitly described families of preminimal affinizations (each family contains more than one equivalence class of affinizations). One of such families consist of preminimal affinizatons of order $1$ (family (c) in the notation of \cite{cp:minirr}). The elements belonging to this class are shown to be minimal affinizations in  \cite{cp:minirr}. Moreover, they are not comparable to any element of the other two families which are formed by the preminimal affinizations of order $2$: the coherent ones and the incoherent ones (families (a) and (b) in the notation of \cite{cp:minirr}). It was conjectured in \cite{cp:minirr} that a given coherent preminimal affinization was equivalent to its incoherent counterpart. This would imply that all members of all three families were minimal affinizations, thus completing the classification. Theorem \ref{t:cohnotmin} shows that the coherent preminimal affinizations listed in  \cite{cp:minirr} are actually not minimal affinizations. In fact, the proof of  Theorem \ref{t:cohnotmin} will show that a given coherent preminimal affinization is strictly larger than its incoherent counterpart in the partial order of affinizations. Thus, all elements of the incoherent family are minimal affinizations. The classification of irregular minimal affinizations for type $D_4$ can then be summarized as: the 3 equivalent classes of incoherent preminimal affinizatios of order $2$ together with the preminimal affinizations of order $1$ listed in family (c) of the main Theorem of \cite{cp:minirr}.  See also Remark \ref{r:cpconj} for comments related to the structure of these affinizations including an explanation of the erroneous announcement about the correctness of the conjecture from \cite{cp:minirr} made in \cite{mou}.\endd
\end{rem}

The proof of Theorem \ref{t:cohnotmin}, given in Section \ref{ss:cohnotmin}, relies on tensor product results from \cite{cha:braid,hmp:tpa}, which will be reviewd in Section \ref{ss:tp}, and on the computation of certain outer multiplicities for preminimal affinizations satisfying
\begin{equation}\label{e:moIlambda2}
{\rm mo}(\bs{\omega}_{I^\lambda}) =2 \quad\text{and}\quad \lambda(h_{i_*})=0,
\end{equation}
where $\lambda=\wt(\bs{\omega})$, which we now explain. Thus, let $\bs{\omega}\in\mathcal P^+$ be preminimal satisfying \eqref{e:moIlambda2}, let $k\in\partial I$ be such that
\begin{equation}\label{e:Llambnotkmin}
\bs{\omega}_{I^\lambda} \quad\text{is not $k_\lambda$-minimal},
\end{equation}
and set $$V=V_q(\bs{\omega}).$$ One easily checks that $\supp(\lambda)$ intersects both connected components of $I_k\setminus\{i_*\}$. Notice however that we are allowing the possibility
\begin{equation*}
  \supp(\lambda)\cap [k,i_*]=\emptyset,
\end{equation*}
in which case $k_\lambda=k$. Recall \eqref{e:defvtheta} and set
\begin{equation}
\nu=\lambda-\vartheta_{I^\lambda} \qquad\text{and}\qquad  \nu_{l}=\lambda-\vartheta_{I_l^\lambda} \quad\text{for}\quad l\in\partial I.
\end{equation}
The following proposition will be crucial in the proof of Theorem \ref{t:cohnotmin}.

\begin{prop}\label{p:outer}
	In the above notation, we have:
	\begin{enumerate}[(a)]
		\item Let $l\in\partial I$. Then, $m_{\nu_{l}}(V)=\delta_{l,k}$ and $m_{\lambda-\alpha_{l_\lambda}}(V)=0$ if $l_\lambda\in\supp(\lambda)$.
		\item If $\mu\in P^+$ satisfies  $\nu<\mu<\lambda$ and $m_\mu(V)>0$, then $\mu=\nu_k$,
		\item\label{p:outerc} If $\bs{\omega}_{I^\lambda}$ is coherent and $k_\lambda\in\supp(\lambda)$, $m_\nu(V)=1$,
		\item\label{p:outeri} If $\bs{\omega}_{I^\lambda}$ is incoherent and $k_\lambda\in\supp(\lambda)$, $m_\nu(V)=0$.\endd
	\end{enumerate}
\end{prop} 

Since $\bs{\omega}_{I^\lambda}$ is $l$-minimal for $l\ne k$ and $I_l$ is of type $A$, the equality $m_{\nu_{l}}(V)=0$ is immediate from the well-known Lemma \ref{lema multi subdiagrama} below which also implies $m_{\lambda-\alpha_{l_\lambda}}(V)=0$ for all $l\in\partial I$ such that $l_\lambda\in\supp(\lambda)$. The remaining statement of part (a) (the equality $m_{\nu_{k}}(V)=1$) will be proved in Section \ref{ss:KR}.  

Using part (a), part (b) is then easily proved as follows. The condition $\mu>\nu$, together with \eqref{e:Jmu},  implies that $\mu=\lambda-\vartheta_J$ for some connected subdiagram $J$ properly contained in $I^\lambda$. One easily checks that, for such $J$, we have 
\begin{equation}
  \lambda-\vartheta_J \in P^+ \qquad\Leftrightarrow\qquad J=[l_\lambda,m_\lambda] \quad\text{with}\quad l,m\in\partial I, \ l_\lambda,m_\lambda\in\supp(\lambda).
\end{equation}
Hence, $\mu=\nu_l$ for some $l\in\partial I$ or $\mu=\lambda-\alpha_{l_\lambda}$ with $l_\lambda\in\supp(\lambda)$ and part (a) implies $\mu=\nu_k$.

Note that parts (a) and (b) of Proposition \ref{p:outer} imply that
\begin{equation*}
m_\nu(V) = \dim(V_\nu) - \dim(V_q(\lambda)_\nu) - \dim(V_q(\nu_k)_\nu).
\end{equation*}	
Hence, proving parts (c) and (d) is equivalent to proving that
\begin{equation}\label{e:4n-2}
\dim(V_\nu) = \dim(V_q(\lambda)_\nu) + \dim(V_q(\nu_k)_\nu) + \xi
\end{equation}	
where $\xi=1$ for part (c) and $\xi=0$ for (d). The proof of \eqref{e:4n-2} will be given in Sections \ref{ss:tpinc} and \ref{ss:tpcoh} using qcharacter theory.

Part of the hypotheses on Theorem \ref{t:cohnotmin} is explained by the following lemma. The remaining hypotheses come from the generality that Theorem \ref{t:tpa}  below is presently proved. 

\begin{lem}\label{l:<nuk}
	Let $V$ be as in Proposition \ref{p:outer} and assume either one of the following hypothesis:
	\begin{enumerate}[(i)]
		\item $\lie g$ is of type $D$,
		\item $I_k^\lambda$ is of type $A_3$,
		\item $\lie g$ is of type $E_6$ and $\supp(\lambda)=\partial I$.
	\end{enumerate}	
	If $\mu\in P^+$, satisfies  $m_\mu(V)>0$ and $\mu\ne \lambda$, then $\mu\le\nu_k$.\endd
\end{lem}

The conclusion of this lemma is false outside these hypotheses. However, although we use it in a strong manner in the proof of Theorem \ref{t:cohnotmin}, these hypotheses may not be necessary conditions for the validity of Conjecture \ref{cj:cohnotmin}. In fact, we believe the approach we use here can be carried out in broader generality and the first step is to replace this lemma by a characterization of the maximal elements of the set $\{\mu\in P^+:\mu<\lambda, m_\mu(V)>0\}$. We will address this in a future work. 
In the case that hypothesis (ii) holds, Lemma \ref{l:<nuk} is a simple consequence of Lemma \ref{l:many0}. More generally, it will be proved as a consequence of Lemma \ref{l:remroots}.

\subsection{Graded Limits}\label{ss:grlim}
Let $\lambda\in P^+$, and suppose $\bs{\omega}\in\mathcal P^+_q$ is of the form 
\begin{equation}
\bs\omega=\prod_{i\in I} \bs\omega_{i,a_i,\lambda(h_i)} \quad\text{for some}\quad a_i\in\mathbb F^\times.
\end{equation}
Suppose further that $a_i/a_j\in q^{\mathbb Z}$ for all $i,j\in I$ (which is the case if $V_q(\bs{\omega})$ is a minimal affinization). In that case,   there exists a $\lie g[t]$-module $L(\bs{\omega})$, referred to as the graded limit of $V_q(\bs{\omega})$, satisfying
\begin{equation}\label{e:grlimch}
\ch(L(\bs{\omega})) = \ch(V_q(\bs{\omega})).
\end{equation}
The construction of $L(\bs{\omega})$ and the related literature will be revised in Section \ref{ss:classlim}.

The graded local Weyl module of highest weight $\lambda$ is the $\lie g[t]$-module $W(\lambda)$ generated by a vector $w$ satisfying the defining relations
\begin{equation*}
\lie n^+[t]w = \lie h[t]_+ w = 0, \quad hw=\lambda(h)w, \quad (x_i^-)^{\lambda(h_i)+1}w=0
\end{equation*}
for all $h\in\lie h$ and $i\in I$. It is known that $W(\lambda)$ is finite-dimensional and any finite-dimensional graded $\lie g[t]$-module generated by a highest-weight vector of weight $\lambda$ is a quotient of $W(\lambda)$ (see \cite{cfk:cat} and references therein). 

Assume $\lie g$ is of type $D$ or $E$ and, for $k\in\partial I$, let $M_k(\lambda)$ be the quotient  of $W(\lambda)$ by the submodule generated by
\begin{equation}\label{e:defMk}
x^-_{\vartheta_J,1}w \quad\text{with}\quad J\subseteq I_i,\ \ \forall \ i\in \partial I_k.
\end{equation} 
The following lemma will be proved in Section \ref{ss:classlim}.

\begin{lem}\label{l:M>>L}
	Let $\bs{\omega}\in\mathcal P_q^+$ and suppose $k\in\partial I$ is such that $V_q(\bs{\omega}_{I_i})$ is a minimal affinization for $i\in \partial I_k$. Then, $L(\bs{\omega})$ is a quotient of $M_k(\lambda)$. In particular, 
	\begin{equation*}\label{e:uppbound}
	m_\mu(V_q(\bs{\omega}))\le m_\mu(M_k(\lambda)) \quad\text{for all}\quad \mu\in P^+.
	\end{equation*}\endd
\end{lem}

\begin{con}\label{cj:grlim}
	Let $\bs{\omega}\in\mathcal P_q^+$ be preminimal of order $2$ and let $k\in\partial I$ be such that $\bs{\omega}$ is not $k$-minimal. If $\bs{\omega}$ is coherent and $\supp(\lambda)\cap[k,i_*]\ne\emptyset$, then $L(\bs{\omega})\cong M_k(\lambda)$.\endd
\end{con}

This is a partial rephrasing of a conjecture from \cite{mou} which was proved still in \cite{mou} for $\lie g$ of type $D_4$ and in \cite{naoi:grlim,naoi:D} in the case that $\lie g$ is of type $D$,  $\lambda$ is regular, and $V_q(\bs{\omega})$ is a minimal affinization. The proof for type $D_4$ given in \cite{mou} depends only on the hypothesis that $\bs{\omega}$ is coherent, regardless if $\lambda$ is irregular or if $V_q(\bs{\omega})$ is a minimal affinization and, as far as we can tell, the same should be true for that given in \cite{naoi:grlim,naoi:D}. In particular, these proofs also provide with formulas for computing the graded character of $L(\bs{\omega})$. Parts (a) and (c) of Proposition \ref{p:outer} as well as Lemma \ref{l:<nuk} in the case $\lie g$ is of type $D$ and $\bs{\omega}$ is coherent can then be deduced from such computations. However, one needs much less information about the graded character to prove these statements.   Namely, after Lemma \ref{l:M>>L}, it suffices to prove 
\begin{equation}
m_\mu(M_k(\lambda)) \ne 0 \quad\Rightarrow\quad \mu\le \nu_k,
\end{equation}
\begin{equation}
 \mu\in\{\nu_k,\nu\},\ k_\lambda\in\supp(\lambda) \quad\Rightarrow\quad m_\mu(M_k(\lambda)) = 1 \quad\text{and}\quad m_\mu(L(\bs{\omega}))\ne 0.
\end{equation}
We will actually prove the following slightly stronger result.

For a graded vector space $V$, let $V[s]$ be the $s$-th graded piece. If $V$ is a graded $\lie g[t]$-module, then $V[s]$ is a $\lie g$-submodule of $V$ for every $s$ and we set
\begin{equation*}
m_\mu^s(V) = m_\mu(V[s]), \quad \mu\in P^+.
\end{equation*}

\begin{prop}\label{p:outerM}
	Let $\bs{\omega}\in\mathcal P_q^+$ be preminimal of order $2$ and let $k\in\partial I$ be such that $\bs{\omega}$ is not $k$-minimal. We have:
	\begin{enumerate}[(a)]
		\item $m_{\nu_{k}}^s(M_k(\lambda))=\delta_{s,1}$ and, if $k_\lambda\in\supp(\lambda)$, $m_\nu^s(M_k(\lambda))=\delta_{s,1}$;
		\item $m_{\nu_k}(L(\bs{\omega}))\ne 0$ and, if $\bs{\omega}$ is coherent and $k_\lambda\in\supp(\lambda)$, then $m_\nu(L(\bs{\omega}))\ne 0$;
		\item under the hypotheses of Lemma \ref{l:<nuk}, if $\mu\in P^+$ satisfies  $\mu<\lambda$ and $m_\mu^s(M_k(\lambda))>0$ for some $s\in\mathbb Z$, then $\mu\le\nu_k$.\endd
	\end{enumerate}
\end{prop} 

The first equalities in parts (a) and (b) as well part (c) will be proved in Section \ref{ss:upper}. The second equality in part (a) is a consequence of the second statement of part (b) together with Lemma \ref{l:M>>L} and \eqref{e:uppbnu}. The second statement of part (b) is a consequence of Proposition \ref{p:outer}(c) which will be proved in Section \ref{ss:tpcoh} (see also Remark \ref{r:cpconj}).

It follows from Proposition \ref{p:outerM} that
\begin{equation}
M_k(\lambda)\cong_{\lie g} V(\lambda)\oplus V(\nu_k)\oplus N\oplus \bigoplus_{\substack{\mu<\nu_k\\\mu\ngeq\nu}}
 V(\mu)^{\oplus m_\mu(M_k(\lambda))},
\end{equation}
where
\begin{equation*}
N\cong  \begin{cases} V(\nu),& \text{if } k_\lambda\in\supp(\lambda),\\ 0,& \text{otherwise}. \end{cases}
\end{equation*}
Moreover,
\begin{equation*}
V(\nu_k)\oplus N \subseteq M_k(\lambda)[1].
\end{equation*}
Let $N_k(\lambda)$ be the quotient of $M_k(\lambda)$ by the $\lie g[t]$-submodule generated by $N$. In light of the above results, part (d) of Proposition \ref{p:outer} becomes equivalent to the following lemma.

\begin{lem}\label{l:increl}
	Assume $\bs{\omega}$ is incoherent and not $k$-minimal.  Then, under the hypothesis of Proposition \ref{p:outer}, $L(\bs{\omega})$ is a quotient of $N_k(\lambda)$.\endd
\end{lem}

\begin{rem}\label{r:demazure}
	The theory of $\lie g$-stable Demazure modules plays a prominent role in the study of graded limits of minimal affinizatios. 
	In \cite{lina,naoi:grlim,naoi:D}, it has been proved that the graded limits of minimal affinizations $\lie g$ of classical type or $G_2$ with regular highest weight are generalized Demazure modules. It appears to us that this is no longer the case for the incoherent minimal affinzations as the simplest case does not appear to be even a Chari-Venkatesh module. Understanding the structure of the module $N_k(\lambda)$, which is most likely isomorphic to the graded limit of the incoherent minimal affinizations, from the point of view of Demazure theory is certainly a topic that must be investigated. We shall come back to this in the future.\endd
\end{rem}

\section{Technical Background}\label{ss:tec}

In this section we review the technical background we shall need for proving Proposition \ref{p:outer} and Theorem \ref{t:cohnotmin}.

\subsection{On the Dimensions of Certain Weight Spaces}\label{ss:parts}

Let $p:Q\to\mathbb Z$ be Kostant's partition function. In other words, $p(\eta)$ is the number of ways of writing $\eta$ as a sum of positive roots or, equivalently, 
$$p(\eta)=\#\mathcal P_\eta,$$
where 
\begin{equation*}
\mathcal P_\eta =\left\{ \xi:R^+\to\mathbb Z_{\ge 0}: \eta = \sum_{\alpha\in R^+}\xi(\alpha)\, \alpha\right\}.
\end{equation*}
In particular, $p(\alpha_i)=1$ for all simple roots and $\eta\in Q\setminus Q^+\Rightarrow p(\eta)=0$. 
The PBW theorem implies $\dim(U(\lie n^+)_\eta)=p(\eta)$. In particular, for $\lambda\in\lie h^*$ and $M(\lambda)$ the Verma module of highest-weight $\lambda$, we have
\begin{equation*}
\dim(M(\lambda)_{\lambda-\eta}) = p(\eta).
\end{equation*}
In the proof of \eqref{e:4n-2}, we will use a similarly flavored formula which applies to $V(\lambda), \lambda\in P^+$, for certain $\eta\in Q^+$.\footnote{Although we are assuming throughout the text that $\lie g$ is simply laced, Section \ref{ss:parts} is valid in complete generality with no need of modifications in the text.}
Thus, consider
$$\mathcal P_\eta^\lambda = \{\xi\in\mathcal P_\eta: \alpha\in\supp(\xi)\Rightarrow \rupp(\alpha)\cap\supp(\lambda)\ne\emptyset\}$$
where
\begin{equation*}
\supp(\xi)=\{\alpha\in R^+:\xi(\alpha)\ne 0\}.
\end{equation*}
Let $v$ be a highest-weight vector for $V(\lambda)$ and recall that, for all  subdiagram $J\subseteq I$, 
\begin{equation}\label{e:nosupvan}
J\cap\supp(\lambda)=\emptyset \quad\text{and}\quad \rupp(\alpha)\subseteq J \quad\Rightarrow\quad x_{\alpha}^-v = 0.
\end{equation}
A straightforward application of the PBW theorem then gives
\begin{equation}\label{e:upb}
\dim(V(\lambda)_{\lambda-\eta}) \le \#\mathcal P_\eta^\lambda \quad\text{for all}\quad \lambda\in P^+, \eta\in Q.
\end{equation}

\begin{prop}\label{p:parts}
	If $J\subseteq I$ is connected and $\lambda\in P^+$ satisfies $\supp(\lambda)\cap J\subseteq \partial J$, then
	\begin{equation*}
	\dim(V(\lambda)_{\lambda-\vartheta_J}) = \#\mathcal P_{\vartheta_J}^\lambda.
	\end{equation*}
\end{prop}

\begin{proof}
	Since $\dim(V(\lambda)_{\lambda-\eta})=	\dim(V(\lambda_J)_{\lambda_J-\eta_J})$ if $\eta\in Q_J$, we	may assume $J = I$.
 It is well known that we have an isomorphism of $\lie n^-$-modules
 \begin{equation*}
   V(\lambda) \cong U(\lie n^-)/U_\lambda \qquad\text{with}\qquad U_\lambda =\sum_{i\in I} U(\lie n^-)(x_i^-)^{\lambda(h_i)+1}.
 \end{equation*} 
 Setting
 \begin{equation*}
   \lie n^-_\lambda = \bigoplus_{\substack{\alpha\in R^+:\\ \rupp(\alpha)\cap\supp(\lambda)\ne \emptyset}} \lie g_{-\alpha},
 \end{equation*}	
 it follows from the PBW Theorem that we have an isomorphism of vector spaces
 \begin{equation*}
 V(\lambda)_{\lambda-\vartheta} \cong \left(U(\lie n^-_\lambda)/U'_\lambda\right)_{-\vartheta} \qquad\text{with}\qquad U'_\lambda =\sum_{\substack{i\in I:\\ \lambda(h_i)\ne 0}} U(\lie n^-_\lambda)(x_i^-)^{\lambda(h_i)+1}.
 \end{equation*} 
 Since $\left( U'_\lambda\right)_{-\vartheta}=0$ and $\dim(U(\lie n^-_\lambda)_{-\vartheta}) = \#\mathcal P_{\vartheta}^\lambda$, the proposition follows.
\end{proof}
		
Let us make explicit all possible values of $\#\mathcal P_{\vartheta_J}^\lambda$. As in the proof of the proposition, to simplify notation, we assume $J=I$ and, hence, $\supp(\lambda)\subseteq\partial I$. In that case, 
\begin{equation}\label{e:dimVnuall}
  \#\mathcal P_{\vartheta}^\lambda = 
  \begin{cases}
    1,& \text{if } \#\supp(\lambda)=1,\\
    \#\overline{\supp}(\lambda),& \text{if } \#\supp(\lambda)=2,\\
    3(n-2)+1,& \text{if } \#\supp(\lambda)=3 \text{ and } \lie g \text{ is of type } D_n,\\
    4(n-2)-2,& \text{if } \#\supp(\lambda)=3 \text{ and } \lie g \text{ is of type } E_n.
  \end{cases}
\end{equation}
To prove this, we will explicitly describe the elements of $\mathcal P_{\vartheta}^\lambda$. Notice that
\begin{equation*}
	\xi\in\mathcal  P_\vartheta^\lambda \quad\Rightarrow\quad \#\supp(\xi)\le\#\supp(\lambda) \quad\text{and}\quad \xi(\alpha)\le 1 \quad\text{for all}\quad \alpha\in R^+.
\end{equation*}
	Therefore, in order to describe $\xi$, it suffices to describe its support. If $\#\supp(\lambda)=1$, the unique element $\xi\in \mathcal P_\vartheta^\lambda$ is characterized by $\supp(\xi)=\{\vartheta\}$.  If $\supp(\lambda)=\{k,l\}$ with $k\ne l$, then, for each $i\in[k,l]$ let $\xi_i$ be the element whose support is $$\{\vartheta_{[k,i]},\vartheta_{(i,l]}\}\setminus\{0\}.$$
	One easily checks that $P_\vartheta^\lambda =\{\xi_i:i\in [k,l]\}$, which proves \eqref{e:dimVnuall} in this case. Finally, assume $\#\supp(\lambda)=3$ and write $\partial I=\{k,l,m\}$ such that $\{m\}$ is a connected component of $I\setminus\{i_*\}$ and $\#[l,i_*]\le \#[k,i_*]$. In particular, $I_m=I\setminus\{m\}=[k,l]$ and, for type $D$, $l$ and $m$ are spin nodes. For any connected subdiagram $I'\subseteq I$ and $i\in I'$, set
	\begin{gather*}
		\mathscr P(I') = \{J\subseteq I': J \text{ is connected}\},\quad \mathscr P_i(I') = \{J\in\mathscr P(I'): i\in J\},\\ \text{and}\qquad \mathscr P_i^o(I') = \mathscr P_i(I') \cup \{\emptyset\}.
	\end{gather*}
	In the case that $I'=I$ we may simply write $\mathscr P$ and $\mathscr P_i$. Note that
	\begin{equation}
	\#\mathscr P_k(I_m) = n-1 \qquad\text{and}\qquad \#\mathscr P_k^o([k,i_*)) = 
	\begin{cases}
	n-2, &\text{for type } D,\\ n-3, &\text{for type }E.
	\end{cases}
	\end{equation}
	Given $J\in\mathscr P_k(I_m)$, let $\xi_J$ be determined by
	\begin{equation*}
	\supp(\xi_J) = \{\alpha_m, \vartheta_J,\vartheta_{I_m\setminus J}\}\setminus\{0\}.
	\end{equation*}
	Given $J\in\mathscr P_k^o([k,i_*))$, let $\xi_J'$ and $\xi_J''$ be determined by
	\begin{equation*}
	\supp(\xi_J') = \{\vartheta_J,\vartheta_{I\setminus [k,m]},\vartheta_{[k,m]\setminus J}\}\setminus\{0\}
	\qquad\text{and}\qquad \supp(\xi_J'') =  \{\vartheta_J,\vartheta_{I\setminus J}\}\setminus\{0\}.
	\end{equation*}
	One easily checks that the elements $\xi_J,\xi_{J'}',\xi_{J'}'', J\in \mathscr P_k(I_m), J'\in \mathscr P_k^o([k,i_*))$ are all distinct. Moreover,
	if $\lie g$ is of type $D$, then
	\begin{equation*}
	\mathcal P_\vartheta^\lambda = \{\xi_J,\xi_{J'}',\xi_{J'}'': J\in \mathscr P_k(I_m), J'\in \mathscr P_k^o([k,i_*))\}, 
	% \qquad\text{and}\qquad \#\mathcal P_\vartheta^\lambda = 3(n-2)+1.
	\end{equation*}
	which proves \eqref{e:dimVnuall}. Consider also $\xi_J''', J\in\mathscr P_k^o([k,i_*))$, determined by
	\begin{equation*}
	\supp(\xi_J''') =  \{\vartheta_J,\alpha_l,\vartheta_{I\setminus (J\cup\{l\})}\}\setminus\{0\}.
	\end{equation*}
	If $\lie g$ is of type $D$, we have $\xi_J'=\xi_J'''$ for all $J\in\mathscr P_k^o([k,i_*))$. However, for type $E$, these are actually new elements and one easily checks that 
	\begin{equation*}
	\mathcal P_\vartheta^\lambda = \{\xi_J,\xi_{J'}',\xi_{J'}'',\xi_{J'}''': J\in \mathscr P_k(I_m), J'\in \mathscr P_k^o([k,i_*))\},
	% \qquad\text{and}\qquad \#\mathcal P_\vartheta^\lambda = 4(n-2)-2,
	\end{equation*}
	completing the proof of \eqref{e:dimVnuall}.

It will  be useful to compare $\dim(V(\lambda)_{\lambda-\vartheta})$ with $\dim(W_{\lambda-\vartheta})$ where
\begin{equation}\label{e:Wparts}
	W=\bigotimes_{i\in\partial I} V(\lambda_i), \qquad \lambda_i=\lambda(h_i)\omega_i,
\end{equation}
and we keep assuming $\supp(\lambda)\subseteq\partial I$. 
We will see that
\begin{equation}\label{e:dimWnu}
\dim(W_{\lambda-\vartheta}) = \dim(V(\lambda)_{\lambda-\vartheta}) + m \quad\text{with}\quad m =
\begin{cases}
0,& \text{if } \#\ \supp(\lambda)=1,\\ 1,& \text{if } \#\ \supp(\lambda)=2,\\n+1,& \text{if } \#\ \supp(\lambda)=3.
\end{cases}
\end{equation}

Let $\mathscr J^\lambda$ be the set of families $J=(J_i)_{i\in\supp(\lambda)}$ of disjoint connected subdiagrams of $I$  satisfying
\begin{equation*}
i\notin J_i \quad\Leftrightarrow\quad J_i=\emptyset
\end{equation*} 	
and, given $\eta\in Q$, set
\begin{equation*}
\mathscr J_\eta^\lambda = \left\{J\in\mathscr J^\lambda:\eta=\sum_{i\in\supp(\lambda)} \vartheta_{J_i}\right\}.
\end{equation*}
One easily sees that
\begin{equation}\label{e:dimWnuparts}
\het_i(\eta)\le 1 \quad\text{for all}\quad i\in I \qquad\Rightarrow\qquad \dim(W_{\lambda-\eta}) = \#\mathscr J_{\eta}^\lambda.
\end{equation}
Consider the map $\varPsi:\mathscr J^\lambda_\vartheta\to \mathcal P_\vartheta^\lambda$ determined by 
\begin{equation*}
\supp(\varPsi(J)) = \{\vartheta_{J_i}:i\in\partial I\}\setminus\{0\},
\end{equation*}
which is clearly surjective. We claim that, for all $\xi\in \mathcal P_\vartheta^\lambda$, we have
\begin{equation}\label{e:supxilambda}
 \#\ \varPsi^{-1}(\xi)=\varPhi(\xi)+1 \qquad\text{where}\qquad \varPhi(\xi) = \#\ \supp(\lambda)-\#\ \supp(\xi).
\end{equation}
Assuming this, we complete the proof of \eqref{e:dimWnu} as follows. If $\#\ \supp(\lambda)=1$, $\varPhi(\xi)=0$ for all $\xi\in \mathcal P_\vartheta^\lambda$. In other words, $\varPsi$ is bijective and  \eqref{e:dimWnu} follows from Proposition \ref{p:parts} and \eqref{e:dimWnuparts}.
If $\#\ \supp(\lambda)=2$, there is a unique $\xi\in \mathcal P_\vartheta^\lambda$ such that $\varPhi(\xi)\ne 0$: the one whose support is $\{\vartheta\}$. Therefore, $\#\ \mathscr J^\lambda_\vartheta = 1+\#\ \mathcal P_\vartheta^\lambda$ and \eqref{e:dimWnu} follows. Finally, If $\#\ \supp(\lambda)=3$, we have to count the sets
\begin{equation*}
\{\xi\in\mathcal P_\vartheta^\lambda: \varPhi(\xi) = 1\} \qquad\text{and}\qquad \{\xi\in\mathcal P_\vartheta^\lambda: \varPhi(\xi) = 2\}.
\end{equation*}
The second set has exactly one element: the one whose support is $\{\vartheta\}$. Therefore, we are left to show that the first set has $n-1$ elements.
But indeed, $\xi$ belongs to that set if, and only if, there exists $i\in I\setminus\{i_*\}$ such that
\begin{equation*}
\supp(\xi) = \{\vartheta_{[i,\partial i]}, \vartheta_{I\setminus[i,\partial i]}\}
\end{equation*}
where $\partial i$ is the element of $\partial I$ lying in the same connected component of $I\setminus\{i_*\}$ as $i$.

It remains to prove \eqref{e:supxilambda}. Fix $J\in\varPsi^{-1}(\xi)$. If $\varPhi(\xi)=0$, then $J_i\ne\emptyset$ for all $i\in\supp(\lambda)$ and the claim is clear. 
If $\varPhi(\xi)=1$, then there exist $k,l\in\supp(\lambda)$ such that $k\in J_l$ and, hence, $J_k=\emptyset$. One easily checks that the unique other element of $\varPsi^{-1}(\xi)$ is the one obtained from $J$ by switching $J_k$ and $J_l$. Finally, if $\varPhi(\xi)=2$, we must have $\#\ \supp(\lambda)=3$ and there exists unique $k\in\supp(\lambda)$ such that $J_k\ne\emptyset$. In particular, $\supp(\lambda)\subseteq J_k$ and the other two elements of $\varPsi^{-1}(\xi)$ are obtained from $J$ by moving $J_k$ to any of the other two positions. This completes the proof of \eqref{e:dimWnu}.

Finally, we deduce some information about the outer multiplicities in $W$. Namely, write
\begin{equation*}
W\cong \bigoplus_{\mu\in P^+} V(\mu)^{\oplus m_\mu(W)}
\end{equation*}

\begin{prop}\label{e:upbclassical}
	Let $\mu\in P^+$ be such that $\het_i(\lambda-\mu)\le 1$ for all $i\in I$. Then, $m_\mu(W)\ne 0$ if and only if $\mu=\lambda-\vartheta_J$ with $J = \overline S$ for some $S\subseteq\supp(\lambda), \#S\ne 1$. In that case,  $m_\mu(W)=1$ if $\#S<3$ and  $m_\mu(W)=2$ if $\#S=3$.
\end{prop}

\begin{proof}
	Set $J=\rupp(\lambda-\mu)$. If $J\cap\supp(\lambda)=\emptyset$, then for all $J'\subseteq J$, $\dim(V(\lambda_i)_{\lambda_i-\vartheta_{J'}})=0$ for all $i\in \partial I$ and, hence, $\dim(W_\mu)=0$. If $\#J\cap\supp(\lambda)= 1, i\in\partial I$ , and  $J'\subseteq J$, then $\dim(V(\lambda_i)_{\lambda_i-\vartheta_{J'}})\le 1$ with equality holding if, and only if, $i\in J'$. In particular, $\dim(W_\mu)=\dim(V(\lambda)_\mu)$ and, hence, $m_\mu(W)=0$. Similarly, we conclude that, if each connected component of $J$ intersects $\supp(\lambda)$ in at most one node, then  $m_\mu(W)=0$. 
	
	Let $k,l\in\supp(\lambda), k\ne l$. If $J=[k,l]$, then $\dim(W_\mu)=\dim(V(\lambda)_\mu)+1$ by \eqref{e:dimWnu} and, hence, $m_\mu(W)=1$. This proves the proposition if $I$ has no trivalent node and we can assume $\lie g$ is of types $D$ or $E$. Since $\lie g$ is simply laced and one easily sees that if there exists $j\in\partial J$ such that $j\notin\supp(\lambda)$, then $\mu=\lambda-\vartheta_J\notin P^+$. Hence, we can assume $\#\supp(\lambda)=3$ and $J=I$. 	
	
	It follows from the cases already considered that
	\begin{equation*}
	m_{\lambda-\vartheta_{[i,j]}}(W) = 1 \qquad\text{for all}\qquad i,j\in\supp(\lambda), i\ne j.
	\end{equation*}
	Let $k\in\partial I$. Writing $\nu=\lambda-\vartheta$ and using \eqref{e:dimVnuall} with $I'=(i_*,k]$ in place of $I$ and $(\lambda-\vartheta_{I_k})^{I'}$ in place of $\lambda$, we see that
	\begin{equation}\label{e:dimWnueachk}
	\dim(V(\lambda-\vartheta_{I_k})_\nu) = d(i_*,k).
	\end{equation}
	One easily checks that 
	\begin{equation}\label{e:dimWnuallk}
	   \sum_{k\in\partial I} d(k,i_*) = n-1.
	\end{equation}
	Combining this with \eqref{e:dimWnu} we get
	\begin{align*}
	  \dim(W_\nu) - \dim(V(\lambda)_\nu) - \sum_{k\in\partial I} \dim(V(\lambda-\vartheta_{I_k})_\nu) = 2.
	\end{align*}
	Since no other irreducible factor of $W$ has $\nu$ as weight, we conclude $m_\nu(W)=2$.
\end{proof}

\subsection{Reduction to Diagram Subalgebras}\label{ss:diag}
We now collect several useful technical results related to the action of diagram subalgebras.

\begin{lem}\cite[Lemma 2.4]{cp:minsl}\label{lema multi subdiagrama}
	Suppose $\emptyset \ne J\subseteq I$ defines a connected subdiagram
	of the Dynkin diagram of $\g$, let $V$ be a highest-$\ell$-weight
	module with highest-$\ell$-weight $\bs\omega\in \cal P^+,
	\lambda=\wt (\bs\omega)$, $v\in V_\lambda\setminus\{0\}$, and
	$V_J=U_q(\tlie g_J)v$. Then, $m_\mu(V)=m_{\mu_J}(V_J)$ for all
	$\mu\in\lambda-Q^+_{J}$.\hfill\qed
\end{lem}

Keeping the notation of Lemma \ref{lema multi subdiagrama}, notice that if $V$ is irreducible, then $V_J\cong V_q(\bs\omega_J)$.
Hence,
\begin{equation}\label{eq mult subdiagram}
\bs\nu\in \bs\omega\cal Q_J \; \Rightarrow \;
\dim(V_q(\bs\omega)_{\bs\nu})=\dim(V_q(\bs\omega_J)_{\bs\nu_J}).
\end{equation}

The next lemma is an easy consequence of \cite[Lemma 2.6]{cp:minsl}.

\begin{lem}\label{lema i0}
	Let $i_0\in I$ be such that
	$$I=J_1\sqcup \{i_0\}\sqcup J_2 \;\;\; (\textrm{disjoint union})$$
	where $J_1$ is of type $A$, $J_2\sqcup\{i_0\}$ is connected and
	$c_{jk}=0$ for all $j\in J_1$, $k\in J_2$. Let $\bs\omega\in \cal
	P^+,\lambda=\wt (\bs\omega)$, and suppose $V_q(\bs\omega_{J_1})$ is
	a minimal affinization of $V_q(\lambda_{J_1})$. Let also
	$$\mu=\lambda-\sum_{j\in I\setminus\{i_0\}} s_j\alpha_j \quad
	\text{with}\quad s_j\in \Z_{\ge 0} \quad\text{for all}\quad j\in I\setminus{i_0}.$$
	If $m_\mu(V_q(\bs\omega))>0$, then $s_j=0$ for all $j\in J_1$.\hfill\qed
\end{lem}

\begin{prop}\cite[Proposition 3.3]{cp:minsl}\label{p:Anotmintheta}
	Suppose $\lie g$ is of type $A$ and let $\bs\omega\in \cal P^+$, $\lambda=\wt (\bs\omega)$ be such that
	\begin{enumerate}
		\item $V_q(\bs\omega)$ is not a minimal affinization of $V_q(\lambda)$, and
		\item $V_q(\bs\omega_{I\setminus \{i\}})$ is a minimal affinization of $V_q(\lambda_{I\setminus \{i\}})$ for any $i\in\partial I$.
	\end{enumerate}
	Then, $m_{\lambda-\vartheta}(V_q(\bs\omega))>0$.\hfill\qed
\end{prop}

\begin{lem}\label{l:many0} Suppose $\lie g$ is of type $D$ or $E$, let $\bs{\omega}\in\mathcal P_q^+$ be preminimal, $\lambda=\wt(\bs\omega)$, and  $V=V_q(\bs{\omega})$. Let $\mu\in P^+$ be such that
\begin{equation}\label{e:hiponmu}
\mu < \lambda \qquad\text{and}\qquad m_\mu(V)\ne 0.
\end{equation}	
Then, $J_\mu:= \rupp(\lambda-\mu)$ is connected and  $m_\mu(V)=m_{\mu_{J_\mu}}(V_q(\bs{\omega}_{J_\mu}))$.
Moreover, for each $k\in\partial I$ we have:
	\begin{enumerate}[(a)]
		\item If $m\notin J_\mu$ for some $m\in [i_*,k]$, then $[m,k]\cap J_\mu=\emptyset$. In particular, $i_*\in J_\mu$.
		\item There exists unique $j\in [i_*,k]$ such that  $(j,k]\cap J_\mu=\emptyset$ and $[j,i_*]\subseteq J_\mu$.   
		\item If $\bs{\omega}$ is $k$-minimal, then $j\ne i_*$.
	\end{enumerate}
\end{lem}

\begin{proof}
	Assuming parts (a) and (b), the first two claims of the lemma can be proved as follows. Let $j_k$ be defined as in (b) for each $k\in\partial I$. It is clear from (a) and (b) that $J_\mu = \overline{\{j_k:k\in\partial I\}}$ showing that it is connected. The second claim of the lemma is then immediate from Lemma \ref{lema multi subdiagrama}.
	
	The first claim in part (a) follows from a straightforward application of Lemma \ref{lema i0} with $i_0=m, J_1= (m,k]$, and $J_2=I\setminus [m,k]$.
	For the second, note that, since $i_*\in [i_*,m]$ for all $m\in\partial I$, if we had $i_*\notin J_\mu$, it would follow that $J_\mu=\emptyset$, contradicting the first assumption in \eqref{e:hiponmu}. 	
	For part (b), let $j$ be the element of $J_\mu\cap [i_*,k]$ which is closest to $k$.  Then, part (a) implies that $j$ satisfies the desired properties. 
	
	To prove (c), note that, if $j=i_*$, we would have $(i_*,k]\cap J_\mu=\emptyset$ and, hence, $\mu\in\lambda-Q^+_{I_k}$. Since $I_k$ is of type $A$ and $\mu<\lambda$, we would have $m_{\mu_{J_{k}}}(V_q(\bs\omega_{I_k}))=0$. On the other hand, Lemma \ref{lema multi subdiagrama} would imply that $m_\mu(V)=m_{\mu_{J_{k}}}(V_q(\bs\omega_{I_k}))$, contradicting the second assumption from \eqref{e:hiponmu}.  
\end{proof}

We can now give a proof of Lemma \ref{l:<nuk} under the assumption that hypothesis (ii) is satisfied. 
Recalling the notation there, we have 
\begin{equation*}
V_q(\bs\omega_{I^\lambda_m}) \quad\text{is a minimal affinization for}\quad m\ne k.
\end{equation*}
Defining $j_m, m\in\partial I$, as in Lemma \ref{l:many0}, it follows that $j_m\ne i_*$ for $m\in\partial I_k$. Hypothesis (ii) implies that $I_k^\lambda\subseteq J_\mu$ and, hence, $\mu\le\nu_k$.

\subsection{qCharacters}\label{ss:qchar}
Let $\Z[\cal P]$ be the integral group ring over $\cal P$. Given $\chi\in \Z[\cal P]$, say 
$$\chi = \sum_{\bs\mu\in\cal P}\chi(\bs\mu)\, \bs\mu,$$ 
we identify it with the function $\cal P\to\mathbb Z, \bs\mu\to \chi(\bs\mu)$. Conversely, any function
$\cal P\to\mathbb Z$ with finite support can be identified with an
element of $\Z[\cal P]$. The qcharacter of $V\in \wcal C_q$ is the
element $\qch(V)$ corresponding to the function
\begin{equation*}
\bs\mu\mapsto \dim(V_{\bs\mu}).
\end{equation*}
We set
\begin{equation*}
\wt_\ell(V) =  \{\bs\mu\in\cal P_q:V_{\bs\mu}\ne 0\} \quad\text{and}\quad \wt_\ell(V_\mu) =  \{\bs\mu\in\wt_\ell(V):\wt(\bs{\mu})=\mu\},
\end{equation*}
for all $\mu\in P$.

The Frenkel-Mukhin algorithm \cite{fremuk:qchar} is one of the main tools for computing qcharacters of simple objects of $\widehat{\mathcal C}_q$, although it is not applicable to any such object. From the basic theory leading to the algorithm, we will only need the following result here (a proof can also be found in \cite{cm:qblock}).

\begin{lem}\label{l:qcbasic}
	Let $V\in\widehat{\mathcal C}_q, i\in I$, and $\bs{\varpi}\in\mathcal P_q$. Suppose there exists $v\in V_{\bs{\varpi}}\setminus\{0\}$ satisfying $x_{i,r}^+v=0$ for all $r\in\mathbb Z$ and that $\bs{\varpi}^{\{i\}} = \bs{\omega}_{i,a,m}$ for some $a\in\mathbb F^\times, m>0$. Then,
	\begin{equation}\label{e:qcbasic}
	\bs{\varpi}\bs{\alpha}_{i,aq^{m-1}}^{-1}\in \wtl(V).
	\end{equation}\hfill\qed
\end{lem}

\subsection{Tensor Products}\label{ss:tp}

The algebra $U_q(\tlie g)$ is a Hopf algebra.  We now review the facts about tensor products of objects from $\widehat{\mathcal C_q}$ that we need. 

It is well-known that the tensor product of weight vectors is a weight vector and, hence, if $V,W\in\mathcal C_q$, we have $\ch(V\otimes
W)=\ch(V)\ch(W)$. Although the tensor product of $\ell$-weight vectors is not an $\ell$-weight vector in general, it was proved in  \cite{freres:qchar} (see also \cite{cm:qblock}) that we still have 
\begin{equation}\label{e:qctp}
\qch(V\otimes W)=\qch(V)\qch(W) \qquad\text{for every}\qquad V,W\in\wcal C_q.
\end{equation}

It turns out that tensor products of nontrivial simple objects from $\wcal C_q$ may be simple as well. 
For the proof of Theorem \ref{t:cohnotmin}, we will need some sufficient criteria for the irreducibility of certain tensor products of minimal affinizations which we now recall. The following is the first half of main result of \cite{hmp:tpa}.

\begin{thm}\label{t:tpa}
	Let $\lie g$ be of type $A_n, \lambda\in P^+\setminus\{0\}$, and consider
	\begin{equation*}
	\bs\pi= \prod_{i\le j} \bs\omega_{i,aq^{-p_{i,j}(\lambda)},\lambda(h_i)} \qquad\text{and}\qquad \bs{\pi}' = \bs\omega_{n,b,\eta}
	\end{equation*}
	for some $a,b\in\mathbb F^\times, \eta\in\mathbb Z_{>0}$ where $j=\max\{i\in I:i\in\supp(\lambda)\}$. Then,  $V_q(\bs{\pi})\otimes V_q(\bs{\pi}')$ is reducible if and only if there exist  $s\in\mathbb Z, j'\in\supp(\lambda)$, and $\eta'\in\mathbb Z_{>0}$ such that $b=aq^s$ and either one of the following options hold:
	\begin{enumerate}[(i)]
		\item $\eta'\le\min\{\lambda(h_{j'}),\eta\}$ and $s+\eta+n-j'+2 = -p_{j',j}(\lambda)-\lambda(h_{j'})+2\eta'$;
		\item $\eta'\le\min\{|\lambda|,\eta\}$ and $\lambda(h_j)+n-j+2 = s-\eta+2\eta'$.
		\hfill\qed
	\end{enumerate}
\end{thm}

\begin{rem}\label{r:tpa}
	Note that $V_q(\bs{\pi})$, with $\bs{\pi}$ as in Theorem \ref{t:tpa}, is an increasing minimal affinization. Similar results for decreasing minimal affinizations as well as for tensor products with KR modules associated to the first node can be obtained from   Theorem \ref{t:tpa} by means of duality arguments. The precise statements can be found in \cite[Corollary 4.2.2]{hmp:tpa}. The second half of Theorem \ref{t:tpa} states that when such tensor products are reducible, they are length-two modules and the Drinfeld polynomial of the irreducible factor with lower highest-weight is explicitly described.\endd
\end{rem}

We will also need a criterion that guarantees the irreducibility of tensor products of KR modules associated to nodes in $\partial I$ when $\lie g$ is of type $D$. To deduce it, we begin by recalling some facts about duality (a slightly more complete review was given in \cite[Section 4.1]{hmp:tpa}). 
For any two finite-dimensional $U_q(\tlie g)$-modules $V$ and $W$, we have
\begin{equation}\label{e:tpdual}
(V\otimes W)^*\cong W^*\otimes V^*.
\end{equation}
Also, given $\bs\omega\in\cal P^+$, we have
 \begin{equation}\label{e:drhw}
 V_q(\bs\omega)^*\cong V_q(\bs\omega^*) \qquad\text{where}\qquad \bs\omega^*_i(u) = \bs\omega_{w_0\cdot i}(q^{-h^\vee}u).
 \end{equation}
 Here,  $h^\vee$ is the dual Coxeter number of $\lie g$, $w_0$ is the longest element of $\mathcal W$ and $w_0\cdot i =j$  if and only if $w_0\omega_i=-\omega_j$. 
The following lemma is well-known and easily established.

\begin{lem}\label{l:irrbyhw}
	Suppose $V$ is an object from $\wcal C_q$. Then, $V$ is simple if and only if both $V$ and $V^*$ are highest-$\ell$-weight modules.\hfill\qed
\end{lem}

The following is a rewriting of part of \cite[Corollary 6.2]{cha:braid}.

\begin{prop}
	Suppose $\lie g$ is of type $D_n$, $i,j\in\partial I$ be distinct $m_i,m_j\in\mathbb Z_{>0}, a_i,a_j\in\mathbb F^\times$, and let $V= V_q(\bs\omega_{i,a_i,m_i})\otimes V_q(\bs\omega_{j,a_j,m_j})$ and $m=\min\{m_i,m_j\}$. The following are sufficient conditions for $V$ to be a highest-$\ell$-weight module:
	\begin{enumerate}[(a)]
		\item $\frac{a_j}{a_i}\ne q^{m_i+m_j+2(2s-p)}$ for all $1\le p\le m,\ 1\le s\le \lfloor\frac{n-1}{2}\rfloor $ if both $i$ and $j$ are spin nodes.
		\item $\frac{a_j}{a_i}\ne q^{m_i+m_j+n-2p}$ for all $1\le p\le m$ if either $i$ or $j$ is not a spin node.	\hfill\qed
	\end{enumerate}
\end{prop}

\begin{rem}
	There is a typo in \cite[Corollary 6.2]{cha:braid} regarding part (a) of the above proposition. Namely, the range for the parameter $s$ is claimed to be $0\le s\le \lfloor\frac{n-1}{2}\rfloor$. The absence of the possibility $s=0$ is crucial for our purposes. We have rechecked the computations related to the proof of  \cite[Corollary 6.2]{cha:braid} and have established that indeed $s=0$ can be removed from the range. Note that this correction is compatible with part (b) of the proposition in the sense that, in type $D_4$, since all elements of $\partial I$ ``are spin nodes'', part (a) should ``coincide'' with (b). If $s=0$ were allowed, the number of obstructions coming from (a) would be twice as many as from part (b). With this correction, parts (a) and (b) coincide in all elements of $\partial I$ for type $D_4$.\endd 
\end{rem}

Recall that, if $\lie g$ is of type $D$, then
\begin{equation}
w_0\cdot i = i\ \text{ if $i$ is not a spin node}.
\end{equation}
In particular, if $i$ is a spin node, so is $w_0\cdot i$.
Then, combining the last proposition with \eqref{e:tpdual}, \eqref{e:drhw}, and Lemma \ref{l:irrbyhw}, one easily establishes:

\begin{cor}\label{c:tpd}
	Suppose $\lie g$ is of type $D_n$, $i,j\in\partial I$ be distinct $m_i,m_j\in\mathbb Z_{>0}, a_i,a_j\in\mathbb F^\times$, and let $V= V_q(\bs\omega_{i,a_i,m_i})\otimes V_q(\bs\omega_{j,a_j,m_j})$ and $m=\min\{m_i,m_j\}$. The following are sufficient conditions for $V$ to be irreducible.
	\begin{enumerate}[(a)]
		\item $(\frac{a_j}{a_i})^{\pm 1}\ne q^{m_i+m_j+2(2s-p)}$ for all $1\le p\le m,\ 1\le s\le \lfloor\frac{n-1}{2}\rfloor $ if both $i$ and $j$ are spin nodes.
		\item $(\frac{a_j}{a_i})^{\pm 1}\ne q^{m_i+m_j+n-2p}$ for all $1\le p\le m$ if either $i$ or $j$ is not a spin node.\hfill\qed
	\end{enumerate}
\end{cor}

\begin{rem}
	Using the combinatorics of qcharacters in terms of tableaux, a necessary and sufficient condition in the context of part (a) of the above Corollary was obtained in \cite{fer:thesis}. Moreover,  in the case that $V$ is reducible, an explicit description of the Drinfeld polynomial of its irreducible factor whose highest weight is the second highest was also obtained. Comments about the difference between the sufficient condition given by Corollary \ref{c:tpd} and the necessary and sufficient one obtained in \cite{fer:thesis} will appear in \cite{hmp:tpd}.  For the moment, it suffices to say that $s=0$ (see previous remark) indeed corresponds to an irreducible tensor product according to \cite{fer:thesis}. \endd
\end{rem}

\subsection{Classical and Graded Limits}\label{ss:classlim} Let $\mathbb A=\mathbb C[q,q^{-1}]\subseteq\mathbb F$ and let $U_{\mathbb A}(\tlie g)$ be the $\mathbb A$-subalgebra of $U_q(\tlie g)$ generated by the elements $(x_{i,r}^\pm)^{(k)}, k_i^{\pm 1}$ for $i\in I,r\in\mathbb Z$, and $k\in\mathbb Z_{\ge 0}$ where  $(x_{i,r}^\pm)^{(k)} = \frac{(x_{i,r}^\pm)^k}{[k]!}$. Define $U_{\mathbb A}(\lie g)$ similarly and notice that $U_{\mathbb A}(\lie g)=U_{\mathbb A}(\tlie g)\cap U_q(\lie g)$ . For the proof of the next proposition see \cite[Lemma 2.1]{cha:fer} and the locally cited references.

\begin{prop}
	We have $U_q(\tlie g)=\mathbb F\otimes_{\mathbb A} U_{\mathbb A}(\tlie g)$ and $U_q(\lie g)=\mathbb F\otimes_{\mathbb A} U_{\mathbb A}(\lie g)$.\endd
\end{prop}

Regard $\mathbb C$ as an $\mathbb A$-module by letting $q$ act as $1$ and set
\begin{equation}
\overline{U_q(\tlie g)} = \mathbb C\otimes_\mathbb A U_\mathbb A(\tlie g) \qquad\text{and}\qquad \overline{U_q(\lie g)} = \mathbb C\otimes_\mathbb A U_\mathbb A(\lie g).
\end{equation}
Denote by $\overline\eta$ the image of $\eta\in U_\mathbb A(\tlie g)$ in $\overline{U_q(\tlie g)}$. The proof of the next proposition can be found in \cite[Proposition 9.2.3]{cp:book} and \cite{lus:book}.

\begin{prop}\label{p:cluq}
	$U(\tlie g)$ is isomorphic to the quotient of $\overline{U_q(\tlie g)}$ by the ideal generated by ${\overline k_i}-1, i\in I$. In particular, the category of $\overline{U_q(\tlie g)}$-modules on which $k_i$ act as the identity operator for all $i\in I$ is equivalent to the category of all $\tlie g$-modules.\endd
\end{prop}

Denote by $\cal P_\mathbb A^+$ the subset of $\cal P_q$ consisting of $n$-tuples of polynomials with coefficients in $\mathbb A$. Let also $\cal P_\mathbb A^\times$ be the subset of $\cal P_\mathbb A^+$ consisting of $n$-tuples of polynomials whose leading terms are in $\mathbb Cq^{\mathbb Z}\backslash\{0\}=\mathbb A^\times$. Given $\bs\omega\in\cal P_\mathbb A^+$, let $\overline{\bs\omega}$ be the element of $\cal P^+_q$ obtained from $\bs\omega$ by evaluating $q$ at $1$.
Given a $U_\mathbb A(\tlie g)$-submodule $L$ of a $U_q(\tlie g)$-module $V$, define
\begin{equation}\label{e:clm}
\bar L = \mathbb C\otimes_\mathbb A L.
\end{equation}
Then, $\bar L$ is a $\tlie g$-module by Proposition \ref{p:cluq}. The next theorem was proved in \cite{cp:weyl}.

\begin{thm}\label{t:llattice}
	Let $\bs\omega\in \cal P_\mathbb A^\times$, $V\in\tilde{\mathcal C}_q$ a highest-$\ell$-weight module of highest $\ell$-weight $\bs{\omega}$, $v\in V_{\bs{\omega}}\setminus\{0\}$, and $L=U_\mathbb A(\tlie g)v$. Then, $\bar L$ is a highest-$\ell$-weight module for $\tlie g$ with highest-$\ell$-weight $\overline{\bs{\omega}}$ and $\ch(\bar L)=\ch(V)$.\endd
\end{thm}

Given $\bs\omega\in \cal P_\mathbb A^\times$, we denote by $\overline{V_q(\bs\omega)}$ the $\tlie g$-module $\bar L$ with $L$ as in the above theorem.

Assume  $\overline{\bs{\omega}}=\bs{\omega}_{\lambda,a}$ for some $a\in\mathbb C^\times$ and let $\bar v$ be a nonzero vector in $\overline{V_q(\bs{\omega})}_{\lambda}$. It follows from Theorem \ref{t:llattice} that
\begin{equation*}
(h\otimes f(t))\bar v = f(a) h\bar v \quad\text{for all}\quad h\in\lie h,\ f(t)\in\mathbb C[t,t^{-1}].
\end{equation*}
Moreover, it follows from the proof of \cite[Proposition 3.13]{mou} that, if $J$ is a connected subdiagram of type $A$ such that $V_q(\bs{\omega}_J)$ is a minimal affinization, then
\begin{equation}\label{e:minArel}
x_{\alpha,r}^-\bar v = a^rx_{\alpha}^-\bar v \quad\text{for all}\quad \alpha\in R_J^+, r\ge 0.
\end{equation}
We shall regard $\overline{V_q(\bs{\omega})}$ as a $\lie g[t]$-module which is generated by $\bar v$ by Proposition \ref{p:genbycurrent}. Denote by $L(\bs{\omega})$ the $\lie g[t]$-module obtained from $\overline{V_q(\bs{\omega})}$ by pulling-back the action by the automorphism $\tau_a$ defined in \eqref{e:grtaua} and let $v\in L(\bs{\omega})_\lambda\setminus\{0\}$. It follows from the above considerations that $L(\bs{\omega})=U(\lie g[t])v$, $\lie n^+[t]v = \lie h[t]_+ v = 0, hv=\lambda(h)v$ for all $h\in\lie h$. Hence, $L(\bs{\omega})$ is a quotient of $W(\lambda)$. Moreover, Theorem \ref{t:llattice} and part (d) of Theorem \ref{t:ciuqg} imply \eqref{e:grlimch}.
Also, by by \eqref{e:minArel}, if $J$ is a connected subdiagram of type $A$ such that $V_q(\bs{\omega}_J)$ is a minimal affinization,
\begin{equation}\label{e:minAgrel}
x_{\alpha,r}^-v=0 \quad\text{for all}\quad \alpha\in R_J^+,\ r>0,
\end{equation}
which easily implies Lemma \ref{l:M>>L}.

\section{Proofs}\label{s:proof}

\subsection{On Characters and Tensor Products of KR-Modules}\label{ss:KR}

We will need some information on  qcharacters and tensor products of  Kirillov-Resehetikhin modules. 
Thus, let $\bs{\omega}=\bs{\omega}_{i,a,m}$ for some $i\in I,a\in\mathbb F, m>0$ and let $v\in L(\bs{\omega})_{m\omega_i}\setminus\{0\}$.
We begin with the following well-known fact:
\begin{equation}\label{e:ifKR}
\mu\in P^+,\ \mu<m\omega_i,\ m_\mu(L(\bs{\omega}))\ne 0 \quad\Rightarrow\quad \het_i(\lambda-\mu)>1.
\end{equation}
Since, for all connected subdiagrams $J$, we have $\het_i(\vartheta_J)\le 1$, this implies
\begin{equation}\label{e:KRupb}
\dim(V_q(\bs{\omega})_{m\omega_i-\vartheta_J}) = \dim(V(m\omega_i)_{m\omega_i-\vartheta_J}) \le 1 
\end{equation}
where the inequality follows from Proposition \ref{p:parts} which also implies that 
\begin{equation}\label{e:KRdim}
\dim(V_q(\bs{\omega})_{m\omega_i-\vartheta_J}) =  1 \quad\Leftrightarrow\quad i\in J. 
\end{equation}

Let $M$ be a tensor product of KR-modules associated to distinct nodes, say $V_q(\bs{\omega}_{i,a_i,m_i}),i\in I$, and set $\lambda=\sum_i m_i\omega_i$.
Let also $W= \otimes_{i\in I} V(m_i\omega_i)$. 
It follows from the above discussion that
\begin{equation}\label{e:tpKR=tpc}
 \dim(W_{\lambda-\vartheta_J}) = \dim\left(M_{\lambda-\vartheta_J}\right) \qquad\text{for all}\qquad J\subseteq I.
\end{equation}
In particular, Proposition \ref{e:upbclassical} apply to $M$ in place of $W$. Therefore, if $V,\nu_k$, and $\nu$ are as in Proposition \ref{p:outer}, it follows that
\begin{equation}\label{e:uppbnunuk}
 m_{\nu_k}(V)\le 1 \qquad\text{and}\qquad m_\nu(V)\le 2.
\end{equation}
Proposition \ref{p:Anotmintheta} then implies that $m_{\nu_k}(V)=1$ thus completing the proof of part (a) of Proposition \ref{p:outer}.
A proof of \eqref{e:ifKR} will be reviewed along the way when we perform some estimates using graded limits in Section \ref{ss:upper}. These estimates will also imply that we actually have
\begin{equation}\label{e:upbnu}
m_\nu(V)\le 1,
\end{equation}
an improvement of \eqref{e:uppbnunuk} which will be crucial in our approach for proving the last two parts of Proposition \ref{p:outer}.

Since most of the literature on qcharacters uses the $Y$-notation of \cite{freres:qchar}, we shall write all arguments within the context of qcharacters using that notation as well. Given $a\in\mathbb F^\times, i\in I, r\in\Z, m\in \Z_{\ge 0}$, set
\begin{equation}\label{e:Ynotation}
Y_{i,r,m} = \bs\omega_{i,aq^{m+r-1},m}.
\end{equation}
Let $\cal P_\mathbb Z$ be the submonoid of $\mathcal P_q$ generated by $Y_{i,r}:=Y_{i,r,1},i\in I,r\in\mathbb Z$.
Following \cite{fremuk:qchar}, given $\bs\omega\in\cal P_\mathbb Z\setminus\{\bs 1\}$, define
\begin{equation}\label{eq def r(omega)}
r(\bs\omega):=\max\{r\in \Z : Y_{i,r}^{\pm 1} \text{ appears in } \bs\omega \text{ for some } i\in I\}.
\end{equation}
Then, $\bs\omega$ is said to be right negative if $Y_{i,r(\bs\omega)}$ does not appear in $\bs\omega$ for all $i\in I$.
Clearly, the product of right negative $\ell$-weights is a right negative $\ell$-weight and a dominant $\ell$-weight is not right negative.

Given a connected subdiagram $J\subseteq I$ define
\begin{equation*}
J^+ = \{i\in I\setminus J: c_{i,j}<0 \text{ for some } j\in J\}.
\end{equation*}
If $l\in J$, define also
\begin{equation}\label{e:partial_l}
\partial_lJ = \begin{cases} l,& \text{ if } J=\{l\},\\ (\partial J)\setminus\{l\},& \text{otherwise,}\end{cases}
\end{equation}
and, 
\begin{equation}\label{e:KR(J)}
\begin{aligned} Y_{l,r,m}(J) = &\ Y_{l,r,m-1} \left(\prod_{i\in J^+} Y_{i,r+2(m-1) + d(i,l)}\right)\left(\prod_{i\in\partial_l J} Y_{i,r+2m + d(i,l)}\right)^{-1}\times \\ & (Y_{i_*,r+2m+d(i_*,l)})^{\epsilon}
\end{aligned}
\end{equation}
for all $r,m\in\mathbb Z, m>0$,  where
\begin{equation*}
\epsilon = 0 \text{ if $J$ is of type $A$ and } \epsilon=1 \text{ otherwise.}
\end{equation*}
In particular, 
\begin{equation*}
Y_{l,r,m}(J) \text{ is right negative\ \ \ and}\quad r(Y_{l,r,m}(J)) = r+2m+d_{l,J}
\end{equation*}
where $d_{l,J}=\max\{d(l,j):j\in J\}$.
Set also
\begin{equation}\label{e:KR(J)top}
Y_{l,r,m}(J) = Y_{l,r,m}  \quad\text{if}\quad l\notin J.
\end{equation}

\begin{lem}\label{l:qcKR1}
	Let $V=V_q(Y_{l,r,m})$ for some $l\in I, r,m\in\mathbb Z,m>0$, and $\lambda=m\omega_l$. For every connected subdiagram $J\subseteq I$ containing $l$,  $\dim(V_{\lambda-\vartheta_J})= 1$ and $V_{\lambda-\vartheta_J}=V_{Y_{l,r,m}(J)}$.
\end{lem}

\begin{proof}
	After \eqref{e:KRupb} and \eqref{e:KRdim}, it suffices to show that
	\begin{equation}\label{e:qcKR1}
	Y_{l,r,m}(J)\in\wtl(V)
	\end{equation}
	for every connected subdiagram $J$ which is either empty or contains $l$. This is obvious if $J=\emptyset$. Otherwise, let $i\in\partial_l J, J'=J\setminus\{i\}$, and assume, by induction hypothesis on $\#J$, that $Y_{l,r,m}(J')\in\wtl(V)$. Note also that $\lambda-\vartheta_{J'}+\alpha_i$ is not a weight of $V$ and, hence, if $v\in V_{Y_{l,r,m}(J')}$, we have
	\begin{equation*}
	x_{i,r}^+v = 0 \qquad\text{for all}\qquad r\in\mathbb Z.
	\end{equation*}
	It follows that the hypotheses of Lemma \ref{l:qcbasic} are satisfied and one easily checks that $Y_{l,r,m}(J)$ is obtained from $Y_{l,r,m}(J')$ by the formula \eqref{e:qcbasic}, thus proving \eqref{e:qcKR1}.
\end{proof}

\begin{lem}\label{l:W(2)_0}
	Let $i,j\in I, \bs{\omega} = Y_{i,r_i,m_i}Y_{j,r_j,m_j}\bs{\varpi}$ for some $r_i,r_j,m_i,m_j\in\mathbb Z, m_i,m_j>0$, and $\bs{\varpi}\in\mathcal P^+$ such that $[i,j]\cap\supp(\wt(\bs{\varpi}))=\emptyset$. Let also $\lambda=\wt(\bs{\omega})$ and assume $r_i\le r_j$. Then,
	\begin{equation*}
	  \dim(V_q(\bs{\omega})_{\lambda-\vartheta_{[i,j]}})= 
	  \begin{cases}
	  d(i,j)+1,& \text{if }  r_j-r_i= 2m_i+d(i,j),\\  d(i,j)+2, &\text{otherwise.}
	  \end{cases}
	\end{equation*}
\end{lem}

\begin{proof}
By Lemma \ref{lema multi subdiagrama}, we may assume $\{i,j\}=\partial I$ and, hence, $\bs{\varpi}=\bs 1$ and $\lie g$ is of type $A$.
If $r_j-r_i= 2m_i+d(i,j)$, then $V_q(\bs{\omega})$ is a minimal affinization. Hence, it is isomorphic to $V(\lambda)$ as $U_q(\lie g)$-module and we are done by Proposition \ref{p:parts} and \eqref{e:dimVnuall}. Otherwise, $V_q(\bs{\omega})$ is not a minimal affinization and Proposition \ref{p:Anotmintheta} implies
\begin{equation*}
 \dim(V_q(\bs{\omega})_{\lambda-\vartheta})\ge \dim(V(\lambda)_{\lambda-\vartheta})+1 = d(i,j)+2.
\end{equation*}
The opposite inequality is immediate from Proposition \ref{e:upbclassical}.
\end{proof}

\begin{rem}
	It is actually not difficult to prove the following improvement of the previous lemma:
	\begin{equation*}
	 \bs{\varpi}Y_{i,r_i,m_i}([i,j]\setminus J)Y_{i,r_2,m_2}(J)\in \wtl(V_q(\bs{\omega}))
	\end{equation*}
	for all connected subdiagram $J$ containing $j$. Moreover, if  $r_j-r_i\ne 2m_i+d(i,j)$, the same holds for $J=\emptyset$. Thus, these are exactly the elements of $\wtl(V_q(\bs{\omega})_{\lambda-\vartheta_{[i,j]}})$ all with multiplicity $1$.\endd
\end{rem}

\subsection{Computations with Graded Limits} \label{ss:upper}
We now compute some upper bounds for outer multiplicities in graded limits which will lead to proofs of Lemma \ref{l:<nuk}, \eqref{e:upbnu}, as well as most of Proposition \ref{p:outerM}.

Given $a\in\mathbb Z_{>0}$ and $i\in I$, set 
\begin{equation*}
R_a = \{\alpha\in R^+:\max\{\het_i(\alpha):i\in I\}=a\}
\end{equation*}
and note that $R^+$ is the disjoint union of the sets $R_a$. Fix $k\in\partial I$, recall the definition of $M_k(\lambda)$ from \eqref{e:defMk}, and let
\begin{equation*}
v\in M_k(\lambda)_\lambda\setminus\{0\}.
\end{equation*} 
We obviously have
\begin{equation}
x_{\alpha,s}^-v =0 \qquad\Rightarrow\qquad x_{\alpha,r}^-v =0 \quad\text{for all}\quad r\ge s.
\end{equation}

\begin{lem}\label{l:remroots1}
	Let $\alpha\in R_1$. Then, $x_{\alpha,2}^-v= 0$. Moreover, if there exists $l\in\partial I_k$ such that either $l_\lambda\notin \rupp(\alpha)$ or $l_\lambda=l\notin\supp(\lambda)$, 	then $x_{\alpha,1}^-v=0$.\endd
\end{lem}

\begin{proof}
	Let $\partial I_k = \{l,m\}$. Then, if $\alpha\in R_1$, we have $\alpha=\beta+\gamma$ with $\rupp(\beta)\subseteq I_l$ and $\rupp(\gamma)\subseteq I_m$. By definition of $M_k(\lambda)$,
	\begin{equation*}
	x_{\beta,1}^-v = x_{\gamma,1}^-v = 0
	\end{equation*}
	and, hence,
	\begin{equation*}
	x_{\alpha,2}^-v = [x_{\beta,1}^-,x_{\gamma,1}^-]v = 0.
	\end{equation*}

	For the second statement, note there exists a connected subdiagram $J\subseteq[l_\lambda,i_*)$ such that 
	\begin{gather*}
	\rupp(\alpha)\subseteq I_l\cup J, \qquad \supp(\lambda)\cap J=\emptyset, \qquad\text{and}\qquad \alpha=\beta+\gamma \qquad\text{for some}\\
	\beta\in R_J^+\cup\{0\} \qquad\text{and}\qquad \gamma\in R_{I_l}^+\cup\{0\}.
	\end{gather*}  
	If either $\beta=0$ or $\gamma=0$, there is nothing else to do. Otherwise, since $J\cap\supp(\lambda)=\emptyset$, $x_\beta^-v = 0$ and it follows that
	\begin{equation}\label{e:x1=0}
	x_{\alpha,1}^-v = [x_{\beta}^-,x_{\gamma,1}^-]v = 0.
	\end{equation}
\end{proof}

Consider the basis $B:=\{x_{\alpha,r}^-: \alpha\in R^+, r\ge 0\}$ of $\lie n^-[t]$. Given a subset $S\subseteq B$ and a choice of total order on $S$, we denote by $U(S)$ the subspace of $U(\lie n^-[t])$ spanned by PBW monomials formed from elements of $S$. Let 
\begin{equation}
R_1^\lambda = \{\alpha\in R_1: I_k^\lambda\subseteq\rupp(\alpha)\}, \qquad R_{>1} = R^+\setminus R_1,
\end{equation}
\begin{equation}
S_1^\lambda = \{x^-_{\alpha,1}:\alpha\in R_1^\lambda\}, \quad\text{and}\quad S_{>1} = \{x^-_{\alpha,r}: \alpha\in R_{>1}, r>0\}.
\end{equation}
Note, by inspecting the root systems, that $\vartheta_{I_k^\lambda}$ is a minimal element of $R_{>1}\cup R_1^\lambda$.
Lemma \ref{l:remroots1} implies that
\begin{equation}
M_k(\lambda) = U(\lie n^-)U(S_{>1})U(S_1^\lambda) v.
\end{equation}
Standard arguments (cf. \cite[Lemma 2.3]{mou}) and the aforementioned minimality of $\vartheta_{I_k^\lambda}$ imply
\begin{equation*}
m_{\nu_k}^s(M_k(\lambda))\le \delta_{s,1}.
\end{equation*}
Then, if $\bs{\omega}\in\mathcal P^+_q$ is $l$-minimal for $l\in\partial I_k$ and $\wt(\bs{\omega})=\lambda$, it follows from Lemmas \ref{l:M>>L} and \ref{lema multi subdiagrama} that
\begin{equation*}
m_{\nu_k}(V_q(\bs{\omega}))\le 1.
\end{equation*}
This recovers the first part of \eqref{e:uppbnunuk} and, moreover, if $\bs{\omega}$ is not $k$-minimal, Proposition \ref{p:Anotmintheta} implies the first statements of parts (a) and (b) of Proposition \ref{p:outerM}. Moreover, the only PBW monomial $x\in U(S_{>1})U(S_1^\lambda)$ such that $xv$ has weight $\nu=\lambda-\vartheta_{I^\lambda}$ is clearly $x=x_{\vartheta_{I^\lambda},1}^-$. Therefore,
\begin{equation}\label{e:uppbnu}
m_\nu^s(M_k(\lambda))\le \delta_{s,1}.
\end{equation}
In particular, if $V$ is as in Proposition \ref{p:outer}, 
\begin{equation}\label{e:4n-2ub}
\dim(V(\lambda)_\nu) + \dim(V(\nu_k)_\nu) \le \dim(V_\nu) \stackrel{\eqref{e:uppbnu}}{\le} \dim(V(\lambda)_\nu) + \dim(V(\nu_k)_\nu) + 1,
\end{equation}	
where the first inequality follows from parts (a) and (b) of Proposition \ref{p:outer}. Regarding the proof of Proposition \ref{p:outerM}, it remains to prove the second statement of part (b) as well as part (c). The remainder of this subsection is dedicated to the latter.

\begin{rem}
	If $V$ is a KR-module associated to the node $i\in I$ and $v$ is the image of its highest-weight vector in the graded limit, a similar argument to the above proves
	\begin{equation*}
	x_{\alpha,1}^-v= 0 \qquad\text{if}\qquad \het_i(\alpha)\le 1
	\end{equation*}
	which implies \eqref{e:ifKR} (cf. \cite{cha:fer,cm:kr,jap:rem,mou,mope:mine6}). \endd
\end{rem}

For the remainder of this subsection, assume the hypotheses of Lemma \ref{l:<nuk}.

\begin{lem}\label{l:remroots}
	Let $\alpha\in R^+$. If there exists $l\in\partial I_k$ such that $l_\lambda\notin \rupp(\alpha)$, $x_{\alpha,1}^-v=0$.\endd
\end{lem}

\begin{proof}
	If $I_k^\lambda$ is of type $A_3$, it follows that $\rupp(\alpha)\subseteq I_l$ and there is nothing to do. Hence, we can assume either hypothesis (i) or (iii) of Lemma \ref{l:<nuk} is satisfied. In particular, $n>4$ and, if $\lie g$ is of type $D$, $k$ is a spin node while $l$ is not a spin node.
	
	Assume $\alpha\in R_a$. Since the case $a=1$ was proved in Lemma \ref{l:remroots1}, we also assume $a\ge 2$. If $\lie g$ is of type $D$,  $\alpha=\beta+\gamma$ with $\beta,\gamma\in R^+$ such that $\rupp(\beta)\subseteq [i_*,l], \rupp(\beta)\cap\supp(\lambda)=\emptyset$, and $\gamma\in R_1$. Since $l_\lambda\notin \rupp(\gamma)$, the case $a=1$ implies that $x_{\gamma,1}^-v=0$ and we are done using \eqref{e:x1=0} once more.
	We are left with the case that hypothesis  (iii) of Proposition \ref{p:outer} is satisfied, i.e., $\lie g$ is of type $E_6$ and $\supp(\lambda)=\partial I$. Recall that, for type $E_6$, we have  
	\begin{equation*}
	R_a\ne\emptyset \quad\Leftrightarrow\quad a\le  3 \qquad\text{and}\qquad a=3 \quad\Rightarrow\quad \rupp(\alpha)=I.
	\end{equation*}
	Hence, we must have $a=2$. An inspection of the root system shows that $\alpha=\beta+\gamma$ with $\beta\in(m,i_*]$ for some $m\in\partial I_k$ and $\gamma\in R_1$ such that $l_\lambda\notin\rupp(\gamma)$. In particular, $\rupp(\beta)\cap\supp(\lambda)=\emptyset$ and \eqref{e:x1=0} completes the proof as before.	
\end{proof}

Let 
\begin{equation}
R^\lambda = \{\alpha\in R^+: I_k^\lambda\subseteq\rupp(\alpha)\} \quad\text{and}\quad S^\lambda_{>1} = \{x^-_{\alpha,r}: \alpha\in R^\lambda\setminus R_1, r>0\}.
\end{equation}
This time we obviously have
\begin{equation}
\vartheta_{I_k^\lambda} = \min R^\lambda
\end{equation}
while Lemma \ref{l:remroots} implies that
\begin{equation}\label{e:remainrootshyp}
M_k(\lambda) = U(\lie n^-)U(S_{>1}^\lambda)U(S_1^\lambda) v.
\end{equation}
Another application of  \cite[Lemma 2.3]{mou} proves
\begin{equation*}
\mu\in P^+,\ \mu<\lambda,\ m_\mu(M_k(\lambda))\ne 0 \quad\Rightarrow\quad \mu\le \nu_k.
\end{equation*}
This, together with Lemmas \ref{l:M>>L} and \ref{lema multi subdiagrama}, implies Lemma \ref{l:<nuk} and part (c) of Proposition \ref{p:outerM}.

\begin{rem}\label{r:cpconj}
	One can proceed with the methods used in the proof of \cite[(5-10)]{mou} to prove that, if $\bs{\omega}_{I^\lambda}$ is coherent, then 
	$$m_{\lambda-s\vartheta_{I^\lambda_k}}(V_q(\bs{\omega}))=m_{\lambda-s\vartheta_{I^\lambda_k}}^s(M_k(\lambda))=1$$ for all $1\le s\le m=\min\{\lambda(h_i): i\in I_k^\lambda\}$ and $m_\nu(V_q(\bs{\omega})) = m_\nu^1(M_k(\lambda))=1$. This would complete the proof of Proposition \ref{p:outerM} as well as of part (c) of Proposition \ref{p:outer}.  However, since the proof of  part (d) of Proposition \ref{p:outer} along the same lines is still unclear to us (cf. Remark \ref{r:demazure}), we will not go in that direction here. Instead, we will give proofs for both cases within the same spirit using qcharacters. 
	
	We recall that \cite[(5-10)]{mou} is  a formula for all outer multiplicities of $L(\bs{\omega})$ for type $D_4$ whose proof implies  the validity of Conjecture \ref{cj:grlim} in that case. It was claimed in the closing remark of \cite{mou} that similar arguments  implied that equation (5-10) also gave the outer multiplicities in the case that $\bs{\omega}$ is incoherent. Part (d) of Proposition \ref{p:outer} implies this is false. The one step that was overlooked in the closing remark of \cite{mou} was the proof of the existence of the second surjective map in the statement of the corresponding incoherent analogue of \cite[Proposition 5.14]{mou}. That map actually does not exist in the case $\bs{\omega}$ is incoherent while it is easily seen to exist if $\bs{\omega}$ is coherent using \cite[Corollary 4.4]{mou}.\endd
\end{rem}

\subsection{Incoherent Tensor Products of Boundary KR-Modules}\label{ss:tpinc}
Continuing our preparation to prove the last two parts of Proposition \ref{p:outer}, we will conduct a partial study of the simple factors of 
\begin{equation}
W=\bigotimes_{i\in\partial I} V_q(Y_{i,r_i,\lambda_i}),
\end{equation}
for certain choices of the parameters  $r_i,\lambda_i$ (recall \eqref{e:Ynotation}). Namely, setting
\begin{equation*}
\bs{\omega} = \prod_{i\in\partial I} Y_{i,r_i,\lambda_i},
\end{equation*}
we study $W$ in the cases that ${\rm mo}(\bs{\omega})=2$ and $\lambda_i\ne 0$ for all $i\in\partial I$. We treat the case that $\bs{\omega}$ is incoherent here, leaving the coherent case to Section \ref{ss:tpcoh}. As usual, we let $k\in\partial I$ be the node such that $\bs{\omega}$ is not $k$-minimal.

The incoherence of $\bs{\omega}$ implies there exists a unique choice of $l,m\in\partial I_k$ such that \footnote{All results of this section remain valid if $\lambda_k=0$ as long as one defines $r_k$ as in the second equality of \eqref{e:inccenterrelr}.}
\begin{equation}\label{e:inccenterrelr}
r_l = r_k+2\lambda_k+d(k,l) \quad\text{and}\quad r_k = r_m+2\lambda_m+d(k,m).
\end{equation}
In particular,
\begin{equation*}
  r_m<r_k<r_l.
\end{equation*}
Recall \eqref{e:KR(J)} and \eqref{e:KR(J)top} and set
\begin{equation}
\bs{\omega}_l = Y_{m,r_m,\lambda_m}(I_l)Y_{k,r_k,\lambda_k}Y_{l,r_l,\lambda_l} \quad\text{and}\quad
\bs{\omega}_m = Y_{m,r_m,\lambda_m}Y_{k,r_k,\lambda_k}(I_m)Y_{l,r_l,\lambda_l}
\end{equation}
Recalling that $I_l=[m,k]$ and $I_m=[l,k]$, one easily checks using \eqref{e:inccenterrelr} that
\begin{gather}\notag
\bs{\omega}_l = Y_{m,r_m,\lambda_m-1}Y_{k,r_k+2,\lambda_k-1} Y_{l,r_l,\lambda_l} Y_{l_*,r_m+2(\lambda_m-1)+d(l_*,m)}\\ \text{and}\\ \notag
\bs{\omega}_m = Y_{m,r_m,\lambda_m}Y_{k,r_k,\lambda_k-1} Y_{l,r_l+2,\lambda_l-1} Y_{m_*,r_k+2(\lambda_k-1)+d(m_*,k)}.
\end{gather}
Here $l_*$ is the element of $(i_*,l]$ closest to $i_*$ and similarly for $m_*$.
In particular, $\bs{\omega}_m,\bs{\omega}_l\in\mathcal P_q^+$. 
Let $\lambda=\wt(\bs{\omega})$ and $\nu=\lambda-\vartheta_I$.

\begin{lem}\label{l:tripletpinc}
	Let $\bs{\varpi}\in\wtl(W)$.
	\begin{enumerate}[(a)]
		\item If $\wt(\bs{\varpi})\ge\nu$, $\dim(W_{\bs{\varpi}})=1$. In particular, $\dim(W_\nu)=\#\wtl(W_\nu)$.
		\item If $\bs{\varpi}\in\mathcal P_q^+\setminus\{\bs{\omega},\bs{\omega}_m,\bs{\omega}_l\}$ and $V_q(\bs{\varpi})$ is an irreducible factor of $W$,  $\wt(\bs{\varpi})\ngeq\nu$. 
	\end{enumerate}
\end{lem}	

\begin{proof}
	Set	
	\begin{equation*}
	D=\{\bs{\mu}\in\wtl(W):\wt(\bs{\mu})\ge \nu\}
	\end{equation*}
	and, as in the proof of \eqref{e:dimWnu}, let $\mathscr J$ be the set of triples $(J_i)_{i\in\partial I}$ of disjoint connected subdiagrams of $I$  satisfying
	\begin{equation*}
	i\notin J_i \quad\Leftrightarrow\quad J_i=\emptyset.
	\end{equation*} 	
	Given $J\in\mathscr J$, define
	\begin{equation}
	\bs\omega(J) = \prod_{i\in\partial I} Y_{i,r_i,\lambda_i}(J_i) \qquad\text{and}\qquad \supp(J)=\{i\in\partial I:J_i\ne\emptyset\}.
	\end{equation}
	It easily follows from \eqref{e:qctp} and Lemma \ref{l:qcKR1} that
	\begin{equation}\label{e:tripletpD}
	D = \{\bs{\omega}(J): J\in\mathscr J\}. 
	\end{equation}
	Thus, part (a) is equivalent to showing that the map $\mathscr J\to\mathcal P_q, J\mapsto\bs{\omega}(J)$, is injective. In preparation for proving that as well as part (b), we first collect some information about the elements $\bs{\omega}(J)$. 
	
	If $\supp(J)=\emptyset$, there is nothing to do. Otherwise, there exists $s\in I$ such that $s\in\partial_aJ_a$ for some $a\in\partial I$. Evidently, given such $s$, $a$ is uniquely determined. Set 
	\begin{equation}\label{e:pi-J-s}
	\bs\pi(J,s)=  Y_{s,r_a+2\lambda_a+d(s,a)}^{-1} \left(\prod_{\substack{b\in\partial I_a:\\ d(s,J_b)=1}} Y_{s,r_b+2(\lambda_b-1)+d(s,b)}\right) \left(\prod_{b\in\partial I} (Y_{s,r_s,\lambda_s-\delta_{s,a}})^{\delta_{s,b}}\right),
	\end{equation}
	where $d(s,J_b)=\min\{d(s,i):i\in J_b\}$ if $J_b\ne \emptyset$ and $d(s,J_b)=\infty$ otherwise.
	By definition, the $s$-th entry of  $\bs{\omega}(J)$ coincides with that of $\bs\pi(J,s)$. 
	We begin by proving that
	\begin{equation}\label{e:tripletpincnotdom}
	  Y_{s,r_a+2\lambda_a+d(s,a)}^{-1} \text{ appears in } \bs{\omega}(J) \text{ unless } s\in\partial I_a \text{ and } (s,a)\in\{(l,k),(k,m)\}.
	\end{equation}
	Indeed, 
	\begin{equation*}
		s\notin\partial I \quad\Rightarrow\quad \bs\pi(J,s) = Y_{s,r_a+2\lambda_a+d(s,a)}^{-1} \left(\prod_{\substack{b\in\partial I_a:\\ d(s,J_b)=1}} Y_{s,r_b+2(\lambda_b-1)+d(s,b)}\right)
	\end{equation*}
	and, hence, $Y_{s,r_a+2\lambda_a+d(s,a)}^{-1}$ does not appear if and only if there exists $b\in\partial I_a$ such that $d(s,J_b)=1$ and
	\begin{equation}\label{e:tripletpincnotdomneg}
	 r_b+2(\lambda_b-1)+d(s,b) = r_a+2\lambda_a+d(s,a).
	\end{equation}
	Note that \eqref{e:tripletpincnotdomneg} implies 
	\begin{equation}\label{e:rb<ra<ra}
	   r_b < r_a+2\lambda_a+d(a,b) \qquad\text{and}\qquad r_a< r_b+2\lambda_b+d(a,b).
	\end{equation}
	Indeed, 	
	\begin{equation*}
	r_b<r_b+2(\lambda_b-1)+d(s,b) \stackrel{\eqref{e:tripletpincnotdomneg}}{=}r_a+2\lambda_a+d(s,a) <r_a+2\lambda_a+d(a,b)
	\end{equation*}
	and
    \begin{equation*}
	r_a<r_a+2\lambda_a+d(s,a) \stackrel{\eqref{e:tripletpincnotdomneg}}{=}r_b+2(\lambda_b-1)+d(s,b) <r_b+2\lambda_b+d(a,b).
	\end{equation*}  
	However, \eqref{e:rb<ra<ra} contradicts \eqref{e:inccenterrelr}, thus proving \eqref{e:tripletpincnotdom} when $s\notin\partial I$. Indeed, the contradiction is clear if  $(a,b)\in \{(l,k),(k,l),(k,m),(m,k)\}$, while  
	\begin{equation*}
	 \begin{aligned}
	 r_l &  \stackrel{\eqref{e:inccenterrelr} }{=}  r_m+2\lambda_m+2\lambda_k+d(k,m)+d(k,l)= r_m+2\lambda_m+2\lambda_k+d(m,l)+2d(k,i_*)\\  & >r_m+2\lambda_m+d(m,l),
	 \end{aligned} 
	\end{equation*}
	revealing the contradiction with \eqref{e:rb<ra<ra}.
   
	 On the other hand, 
	\begin{equation*}
	s\in\partial I \quad\Rightarrow\quad \bs\pi(J,s) =  Y_{s,r_a+2\lambda_a+d(s,a)}^{-1} Y_{s,r_s,\lambda_s-\delta_{s,a}},
	\end{equation*}
	and, hence, there will be a cancellation if, and only if,
	\begin{equation*}
	\exists\ 0\le p< \lambda_s-\delta_{s,a} \quad\text{such that}\quad r_s+2p = r_a+2\lambda_a+d(s,a). 
	\end{equation*}
	But then,
	\begin{equation*}
	 r_s = r_a+2\lambda_a+d(s,a)-2p\le r_a+2\lambda_a+d(s,a),
	\end{equation*}
	and \eqref{e:inccenterrelr} implies that $p=0$ and $(s,a)\in\{(l,k),(k,m)\}$ completing the proof of \eqref{e:tripletpincnotdom}. Moreover,
	\begin{equation}
	  (s,a)\in\{(l,k),(k,m)\} \quad\Rightarrow\quad \bs{\pi}(J,s) = Y_{s,r_s,\lambda_s-1}.
	\end{equation}
	
	Let $J,J'\in\mathscr J$ be such that
	\begin{equation*}
	 \bs{\omega}(J)= \bs{\omega}(J').
	\end{equation*}
	To see that $J=J'$, thus proving part (a), we will show that
	\begin{equation}\label{e:tripletpinc=}
	\partial J_a = \partial J'_a \qquad\text{for all}\qquad a\in\partial I.
	\end{equation}
	Let $s\in\partial_aJ_a$ for some $a\in\partial I$. If $Y_{s,r_a+2\lambda_a+d(s,a)}^{-1}$ appears in $\bs{\omega}(J)$, then $s\in\partial_b J'_b$ for some $b\in\partial I$ and 
	\begin{equation*}
	  r_a+2\lambda_a+d(s,a) = r_b+2\lambda_b+d(s,b).
	\end{equation*} 
	If it were $a\ne b$, this would imply  \eqref{e:rb<ra<ra}, which is a contradiction as seen before. Hence, we must have $b=a$. If $Y_{s,r_a+2\lambda_a+d(s,a)}^{-1}$ does not appear in $\bs{\omega}(J)$, then $(s,a)\in\{(l,k),(k,m)\}$ and $\bs{\pi}(J',s)=Y_{s,r_s,\lambda_s-1}$. The latter implies that $s\in\partial_bJ'_b$ for some $b\in\partial I$ and $(s,b)\in\{(l,k),(k,m)\}$. Hence, $b=a$, completing the proof of \eqref{e:tripletpinc=}.

	We now show that
	\begin{equation}\label{e:D+}
	D\cap\mathcal P_q^+ = \{\bs{\omega},\bs{\omega}_m,\bs{\omega}_l\},
	\end{equation}
	which proves part (b). Let $J\in\mathscr J$ be such that $\bs{\omega}(J)\in\mathcal P_q^+$ and assume $\supp(J)\ne\emptyset$. It follows from \eqref{e:tripletpincnotdom} that
	\begin{equation}
	 \text{either}\quad [k,m]= J_m \quad\text{or}\quad [l,k]= J_k.
	\end{equation}
	For the former, it follows that $J_k=\emptyset$ and
	\begin{equation*}
	 Y_{k,r_k,\lambda_k}(J_k)Y_{m,r_m,\lambda_m}(J_m)\in\mathcal P_q^+.
	\end{equation*}
	Since $Y_{l,r_l,\lambda_l}(J_l)$ is right negative if $J_l\ne\emptyset$, it follows that $\bs{\omega}(J) = \bs\omega_l$. Similarly, if $[l,k]= J_k$, it follows that $\bs{\omega}(J) = \bs\omega_m$.
\end{proof}

 We prove next that
 \begin{equation}\label{e:4n-2final}
 \dim(V_q(\bs{\omega}_l)_\nu) = d(l_*,l)+2 \qquad\text{and}\qquad \dim(V_q(\bs{\omega}_m)_\nu) =  d(m_*,m) + 2,
 \end{equation}
 Plugging $l_*$ in place of of $i$ and $l$ in place of $j$ in Lemma \ref{l:W(2)_0}, the first equality follows provided 
 \begin{equation}\label{e:W(2)_0check}
 r_l - (r_m+2(\lambda_m-1)+d(l_*,m))\ne 2 + d(l_*,l).
 \end{equation}
 But \eqref{e:inccenterrelr} implies the left-hand-side is strictly larger than the right-hand-side.
 For the second statement in  \eqref{e:4n-2final}, we plug $m$ in place of $i$ and $m_*$ in place of $j$ in Lemma \ref{l:W(2)_0} and check, using \eqref{e:inccenterrelr}, that 
 \begin{equation*}
 r_k+2(\lambda_k-1)+d(m_*,k)-r_m\ne 2\lambda_m + d(m_*,m).
 \end{equation*}

\begin{rem}
 	For $\lie g$ of type $D$ and assuming that both elements of $\partial I_k$ are spin nodes, it was proved in \cite[Lemma 4.6.3]{fer:thesis} that  $V_q(\bs{\omega})$ is $\ell$-minuscule. This implies that the qcharacter of $V_q(\bs{\omega})$ can be computed by means of the FM algorithm. A sketchy use of the algorithm was then used to prove  \eqref{e:4n-2final} by performing a counting of $\ell$-weights. The argument presented here, replaces this counting by a combination of Lemma \ref{l:qcbasic}, which is part of the background of the FM algorithm, with the results of Section \ref{ss:parts}. It is interesting to note that the proof of \cite[Lemma 4.6.3]{fer:thesis} also relies on special cases of the results from 
 	Section \ref{ss:parts} whose proof in \cite[Lemma 6.1]{fer:thesis} was sketched by making use of Nakajima's monomial realization of Kashiwara's crystals $B(\lambda), \lambda\in P^+$. The proofs we gave in Section \ref{ss:parts} are completely classical. Although the  strategy developed here for proving Proposition \ref{p:outer}\eqref{p:outeri} does not rely on whether $V_q(\bs{\omega})$ is $\ell$-minuscule or not, it would be interesting to check if this is true in the generality we are working here. To keep the length of the present paper under reasonable limits, we shall leave this topic to a future work.\endd
\end{rem}

 We are ready for:

\begin{proof}[Proof of Proposition \ref{p:outer}(d)]
Since $m_\nu(V) = m_{\nu_{I^\lambda}}(V_q(\bs\omega_{I^\lambda}))$, we can assume $I^\lambda=I$ and, hence, $\supp(\lambda)=\partial I.$
 By \eqref{e:4n-2ub}, we have
\begin{equation*}
  \dim(V_\nu)  = \dim(V(\lambda)_\nu) + \dim(V(\nu_k)_\nu)+\xi \qquad\text{with}\qquad 0\le \xi\le 1,
\end{equation*}
and, after \eqref{e:4n-2}, we need to show that $\xi=0$. 
By \eqref{e:dimWnueachk},
\begin{equation*}
\dim(V(\nu_k)_\nu) = d(k,i_*),
\end{equation*}
while \eqref{e:dimWnu} implies
\begin{equation*}
\dim(W_\nu) = \dim(V(\lambda)_\nu) + n+1.
\end{equation*}
One the other hand, \eqref{e:dimWnuallk} is equivalent to
\begin{equation*}
d(k,i_*) = n-3 - d(l_*,l) - d(m_*,m)
\end{equation*}
and, hence,
\begin{align*}
\dim(W_\nu) - \dim(V_\nu) &= d(l_*,l) + d(m_*,m) + 4 -\xi \\& \stackrel{\eqref{e:4n-2final}}{=}  \dim(V_q(\bs{\omega}_m)_\nu)+\dim(V_q(\bs{\omega}_l)_\nu) - \xi > \dim(V_q(\bs{\omega}_i)_\nu)
\end{align*}
for $i=l,m$. Lemma \ref{l:tripletpinc} implies that $V, V_q(\bs{\omega}_m)$, and $V_q(\bs{\omega}_l)$ are the only possible irreducible factors of $W$ having $\nu$ as a weight. Thus, the above computation shows that both $V_q(\bs{\omega}_m)$ and $V_q(\bs{\omega}_l)$ are indeed irreducible factors of $W$ and, therefore, $\xi=0$.
\end{proof}

\begin{rem}
	At the end of the above proof, we have shown that both $V_q(\bs{\omega}_m)$ and $V_q(\bs{\omega}_l)$ are irreducible factors of $W$. It is interesting to observe that, in the case that $d(l,i_*),d(m,i*)>1$ (in particular $\lie g$ is of type $E$), \eqref{e:uppbnunuk} would suffice in the above proof without the need of the sharper \eqref{e:upbnu} and, hence, independently of the results of Section \ref{ss:upper}.  The same comments apply to the coherent case treated in the next subsection. \endd
\end{rem}

\subsection{Coherent Tensor Products of Boundary KR-Modules}\label{ss:tpcoh}

We recall the following well-known proposition which is easily proved by considering pull-backs by the automorphisms given by \cite[Propositions 1.5 and 1.6]{cha:minr2} together with dualization (cf. \cite[Section 4.1]{hmp:tpa}).
	
	\begin{prop}\label{prop inverte}
		Let $\lambda\in P^+, \bs\omega=\prod_{i\in I}\bs\omega_{i,a_i,\lambda(h_i)}$, and $\bs\varpi=\prod_{i\in I}\bs\omega_{i,b_i,\lambda(h_i)}$, with $a_i,b_i\in
		\C^\times$. If 	there exists $\varepsilon=\pm 1$ such that
		$$\frac{a_i}{a_j}=\left(\frac{b_j}{b_i}\right)^\varepsilon \qquad \text{for all} \qquad i,j\in I,$$
		then $V_q(\bs\omega)\cong_{U_q(\g)}V_q(\bs\varpi)$.\hfill\qedsymbol
	\end{prop}

Let $W$ be as in Section \ref{ss:tpinc} but this time assume $\bs{\omega}$ is coherent. Thus, letting $k\in\partial I$ be the node such that $\bs{\omega}$ is not $k$-minimal, by Proposition \ref{prop inverte}, we may assume
\begin{equation}\label{e:cohcenterrelr}
r_l = r_k+2\lambda_k+d(k,l) \quad\text{for all}\quad l\in \partial I_k.
\end{equation}
Using \eqref{e:KR(J)}, set
\begin{equation*}
\bs{\omega}_l = Y_{k,r_k,\lambda_k}(I_l) \prod_{i\in\partial I_k}Y_{i,r_i,\lambda_i} \qquad\text{and}\qquad \bs{\omega}' = Y_{k,r_k,\lambda_k}(I) \prod_{i\in\partial I_k}Y_{i,r_i,\lambda_i}.
\end{equation*}
Since
\begin{equation*}
d(k,l) - d(k,l_*) = d(l,l_*),
\end{equation*}
\eqref{e:cohcenterrelr} implies
\begin{equation}
\bs{\omega}_l = Y_{k,r_k,\lambda_k-1} Y_{m,r_m+2,\lambda_m-1} Y_{l,r_l,\lambda_l}Y_{l_*,r_l-2-d(l,l_*)},
\end{equation}
where $m\in\partial I_k,m\ne l$, and
\begin{equation}
\bs{\omega}' = Y_{k,r_k,\lambda_k-1} Y_{i_*,r_k+2\lambda_k+d(i_*,k)}\prod_{i\in\partial I_k}Y_{i,r_i+2,\lambda_i-1}.
\end{equation}
If there exists $l\in\partial I_k$ such that 
\begin{equation}\label{e:rlmp}
r_l+2\lambda_l+d(l,m)=r_m+2p \quad \text{for some}\quad 0\le p<\lambda_m,
\end{equation} 
 where $m\in\partial I_k,m\ne l$, set also
\begin{equation*}
\begin{aligned}
\bs{\omega}'' & = Y_{k,r_k,\lambda_k} Y_{l,r_l,\lambda_l}(I_k)Y_{m,r_m,\lambda_m}\\
&  = Y_{k,r_k,\lambda_k}  Y_{k_*,r_l+2(\lambda_l-1) + d(k_*,l)}  Y_{l,r_l,\lambda_l-1}Y_{m,r_m,p}Y_{m,r_m+2(p+1),\lambda_m-p-1}.
\end{aligned}
\end{equation*}
Note that if such $l$ exists, it is unique.

\begin{lem}\label{l:tripletpcoh}
	Let $\bs{\varpi}\in\wtl(W)$.
	\begin{enumerate}[(a)]
		\item If $\bs{\varpi}\in\{\bs{\omega},\bs{\omega}',\bs{\omega}'',\bs{\omega}_i:i\in\partial I_k\}$, $\dim(W_{\bs{\varpi}})=1$.
		\item If $\wt(\bs{\varpi})\ge\nu$ and  $\bs{\varpi}\notin\{\bs{\omega},\bs{\omega}',\bs{\omega}'',\bs{\omega}_i:i\in\partial I_k\}$, then $\bs{\varpi}\notin\mathcal P_q^+$. In particular, if $V_q(\bs{\varpi})$ is an irreducible factor of $W$,  $\wt(\bs{\varpi})\ngeq\nu$. 
	\end{enumerate}
\end{lem}	

\begin{proof}
	Defining $D$ and $\bs{\omega}(J)$ as in the proof of Lemma \ref{l:tripletpinc}, \eqref{e:tripletpD} remains valid. 
	As before, we start by collecting some information about the elements $\bs{\omega}(J)$ such that $\supp(J)\ne \emptyset$. For $\bs\pi(J,s)$ defined as in \eqref{e:pi-J-s}, we will prove that there exists at least one choice of $(a,s)$ with $a\in\supp(J)$ and $s\in\partial_a J_a$ such that 
	\begin{equation}\label{e:Yappears}
		Y_{s,r_a+2\lambda_a+d(s,a)}^{-1}\ \text{ appears in }\ \bs{\omega}(J),
	\end{equation}
unless $\#\supp(J)=1$ and one of the following options holds: 
\begin{enumerate}[(i)]
	\item $J_k=I$; 
	\item $J_k=I_l$ for some $l\in\partial I_k$;	
	\item If $l\in\supp(J), l\ne k$, then $J_l= I_k=[l,m]$ and the pair $(l,m)$ satisfies \eqref{e:rlmp}.
\end{enumerate}

	To prove this, suppose \eqref{e:Yappears} does not hold for all choices of $(a,s)$. Assume first that 
	\begin{equation}
	 \exists\ a\in\supp(J) \quad\text{such that}\quad \partial_aJ_a\cap\partial I_a = \emptyset.
	\end{equation}
	Since none of the options (i)--(iii) satisfy this hypothesis, we need to show this is yields a contradiction. Fix such $a$ and let $s\in\partial_aJ_a$,  which implies $s\notin \partial I_a$. As seen in the proof of  Lemma \ref{l:tripletpinc}, there must exist $b\in\partial I_a$ such that  $d(s,J_b)=1$ for which \eqref{e:tripletpincnotdomneg} holds. As before, this implies \eqref{e:rb<ra<ra} which, this time, contradicts \eqref{e:cohcenterrelr} if $(a,b)\in \{(l,k),(k,l):l\in \partial I_k\}$. In particular,  $a,b\ne k$ and $k\notin\supp(J)$. Note that, if $\{l,m\} = \{a,b\}$ and $t\in\partial_l J_l\cap(i_*,k)$, then $d(t,J_m)>1$, which  implies \eqref{e:Yappears} holds  with $(l,t)$ in place of $(a,s)$. Thus, either $J_a\cup J_b=I_k$ or $k\in J_b$. If $J_a\cup J_b=I_k$, letting $t\in\partial_bJ_b$, it follows that $d(s,t)=1$ and 
	\begin{equation*}
	\bs\pi(J,t)=  Y_{t,r_b+2\lambda_b+d(t,b)}^{-1}  Y_{t,r_a+2(\lambda_a-1)+{d(t,a)}}.
	\end{equation*}
	We claim \eqref{e:Yappears} holds  with $(b,t)$ in place of $(a,s)$, yielding the desired contradiction. Indeed, this is not the case if and only if 
	\begin{equation*}
	r_b+2\lambda_b+d(t,b) = r_a+2(\lambda_a-1)+d(t,a),
	\end{equation*}
	which contradicts \eqref{e:tripletpincnotdomneg}  since $d(s,b)=d(t,b)+1$ and $d(t,a)=d(s,a)+1$. If $k\in J_b$ and there exists $t\in\partial_b J_b\setminus\{k\}$, the same argument yields a contradiction. It remains to deal with the case $J_b=I_a$ and $J_a=I\setminus I_a = (i_*,a]$.  In this case, we check that \eqref{e:Yappears} holds  with $(b,k)$ in place of $(a,s)$. Indeed,
	\begin{equation}\label{e:kpartialother}
	  \bs\pi(J,k)=  Y_{k,r_b+2\lambda_b+d(k,b)}^{-1} Y_{k,r_k,\lambda_k}.
	\end{equation}
	Thus, $Y_{k,r_b+2\lambda_b+d(k,b)}^{-1}$ is canceled if and only if 
	\begin{equation}
	 r_b+2\lambda_b+d(k,b) = r_k + 2p \qquad\text{for some}\qquad 0\le p<\lambda_k.
	\end{equation}
	One easily checks that \eqref{e:cohcenterrelr} implies that there is no such $p$. 
	
	Assume now 
	\begin{equation}\label{e:border-a}
	  \partial_aJ_a \cap\partial I_a \ne  \emptyset  \quad\text{for all}\quad a\in\supp(J)
	\end{equation}
	which implies $\#\supp(J)=1$ and $J_a$ is either $I$ or $I_m$ for some $m\in\partial I_a$. In particular, if $k\in\supp(J)$, then either option (i) or (ii) holds and we are done. Otherwise, we need to show we are in case (iii). Indeed, letting $l$ be the element of $\supp(J)$, we must have $J_l = I$ or $J_l=I_m$ for some $m\in \partial I_l$. If $k\in\partial_lJ_l$, then \eqref{e:kpartialother} holds with $b=l$ and we have a contradiction as before. Hence, we must have $J_l=I_k=[l,m]$ which implies
	\begin{equation*}
	\bs\pi(J,m) =  Y_{m,r_l+2\lambda_l+d(m,l)}^{-1} Y_{m,r_m,\lambda_m}, 
	\end{equation*}
	and we see that $Y_{m,r_l+2\lambda_l+d(m,l)}^{-1}$ is canceled if and only \eqref{e:rlmp} holds, thus proving we are in case (iii). This completes the proof of \eqref{e:Yappears}.	
	
Note that $\bs{\omega}(J)=\bs{\omega}'$ if $J$ is as in (i), $\bs{\omega}(J)=\bs{\omega}_l$ if $J$ is as in (ii), and $\bs{\omega}(J)=\bs{\omega}''$ if $J$ is as in (iii).  Since \eqref{e:Yappears} implies $\bs{\omega}(J)\notin\mathcal P_q^+$ if $J$ does not satisfy one of these three conditions, all claims of the lemma follows. 
\end{proof}

\begin{rem}
	Notice that part (a) of Lemma \ref{l:tripletpcoh} is weaker than that of Lemma \ref{l:tripletpinc}. Indeed, that stronger statement is false in the context of this subsection. Similar arguments to those employed in the proof of Lemma \ref{l:tripletpinc} can be used to show that, if $\wt(\bs\varpi)\ge \nu$, then  $\dim(W_{\bs{\varpi}})\le 2$ and equality holds if and only if there exists $s\in I$ such that
	\begin{equation}\label{e:ra=rbmod}
	 2\lambda_l+d(l,k)+d(s,l) = 2\lambda_m+d(m,k)+d(s,m) \qquad\text{with}\qquad \{l,m\}=\partial I_k
	\end{equation} 
	and one of the following conditions holds:
	\begin{enumerate}[(i)]
		\item There exists such $s$ in $[l,m]$ and $\bs\varpi=\bs\omega(J)$ for some $J\in\mathscr J$ such that $s\in\partial_lJ_l$ and $[m,s)\subseteq J_m$; 
		\item Such $s$ exists only in $(i_*,k]$ and $\bs\varpi=\bs\omega(J)$ for some $J\in\mathscr J$ such that $J_l=[l,s]$ and $J_m=[m,i_*)$. 
	\end{enumerate}
	In case of (i), we have $\bs\varpi = \bs\omega(J')$ with $J'_l=J_l\setminus\{s\}, J'_m=J_m\cup\{s\}$, and $J'_k=J_k$. In case of (ii), $\bs\varpi = \bs\omega(J')$ with $J'_l=[l,i_*), J'_m=[m,s]$, and $J'_k=J_k$.	Since these facts will not play a role in this paper, we omit the details. It might be interesting to observe that, for generic $\lambda$, there is no $s\in I$ satisfying \eqref{e:ra=rbmod} and, hence, part (a) of Lemma  \ref{l:tripletpinc} is valid in the present context for ``most'' values of $\lambda$.\endd
\end{rem}

\begin{lem}\label{l:omega''nofactor}
	If \eqref{e:rlmp} holds, $\bs{\omega}''\in\wtl(V_q(\bs{\omega}))$.
\end{lem}

\begin{proof}
 Observe that $\wt(\bs{\omega}'')=\nu_k$. By \eqref{e:dimWnu} and Proposition \ref{p:parts}, 
		\begin{equation*}
		\dim(W_{\nu_k})=\dim(V(\lambda)_{\nu_k})+1 = d(l,m)+2.
		\end{equation*} 
		On the other hand, Lemma \ref{l:W(2)_0} implies $\dim(V_q(\bs{\omega})_{\nu_k})=d(l,m)+2$, completing the proof of the lemma. To see that 
		Lemma \ref{l:W(2)_0} implies what we claimed, it suffices to check that $p>0$ in \eqref{e:rlmp}. Indeed, we have
		\begin{equation*}
		2p\stackrel{\eqref{e:rlmp}}{=}r_l+2\lambda_l+d(l,m)-r_m \stackrel{\eqref{e:cohcenterrelr}}{=} d(k,l)-d(k,m)+d(l,m)+2\lambda_l.
		\end{equation*}
		Since 
		\begin{equation*}
		d(l,m)=d(k,l)+d(k,m)-2d(k,i_*),
		\end{equation*}
		it follows that
		\begin{equation*}
		2p=2(d(k,l)-d(k,i_*)+\lambda_l)>0.
		\end{equation*}
	
\end{proof}

\begin{proof}[Proof of Proposition \ref{p:outer}(c)]
	Proceeding similarly to the proof of \eqref{e:4n-2final}, this time it follows from Lemma \ref{l:W(2)_0} that
	\begin{equation}\label{e:4n-2coh}
	\dim(V_q(\bs{\omega}_l)_\nu) = d(l_*,l) + 1 \qquad\text{for}\qquad l\in\partial I_k.
	\end{equation}
	Evidently,
	\begin{equation}\label{e:4n-2coh'}
	\dim(V_q(\bs{\omega}')_\nu) = 1.
	\end{equation}
	Write $\partial I_k=\{l,m\}$. Proceeding as in the proof of part (d) of the proposition given in the previous subsection, we get
	\begin{align*}
	\dim(W_\nu) - \dim(V_\nu) &= d(l_*,l) + d(m_*,m) + 4 -\xi \\& =  \dim(V_q(\bs{\omega}')_\nu) + \dim(V_q(\bs{\omega}_m)_\nu)+\dim(V_q(\bs{\omega}_l)_\nu) + 1 - \xi.
	\end{align*}
	Lemmas \ref{l:tripletpcoh}  and \ref{l:omega''nofactor} imply  that $V, V_q(\bs{\omega}_l), V_q(\bs{\omega}_m)$, and $V_q(\bs{\omega}')$ are the only possible irreducible factors of $W$ having $\nu$ as a weight and all of them occur with multiplicity at most $1$. Thus, the above computation shows that all of them are indeed irreducible factors of $W$ as well as $\xi=1$, thus proving \eqref{e:4n-2}.
\end{proof}

\subsection{Proof of Theorem \ref{t:cohnotmin}}\label{ss:cohnotmin}

Let $\bs{\omega}$ be as in Conjecture \ref{cj:cohnotmin}, choose  $m\in\partial I\setminus\{k\}$, and let $\bs{\varpi}\in\mathcal P^+_q$ be such that $\wt(\bs{\varpi})=\lambda$,
\begin{equation*}
\qquad \bs{\omega}_{I_m} = \bs{\varpi}_{I_m}, \quad V_q(\bs{\varpi}_{I_l}) \text{ is minimal for }l\ne k, \quad\text{and}\quad \bs{\varpi} \quad\text{is incoherent.}
\end{equation*}
One easily sees that such $\bs{\varpi}$ is unique for each choice of $m$. To simplify the writing, we will assume further that
\begin{equation}\label{e:mclosetoi*}
d(m,i_*)\le d(l,i_*) \qquad\text{for}\qquad l\in\partial I_k.
\end{equation}
Under the hypothesis of Theorem \ref{t:cohnotmin}, we will show that 
\begin{equation}\label{e:cohincrel}
V_q(\bs{\omega})>V_q(\bs{\varpi}),
\end{equation}
which proves the Theorem. We remark that \eqref{e:cohincrel} holds even if we did not choose $m$ satisfying \eqref{e:mclosetoi*} (in fact, the two choices give rise to equivalent affinizations by Proposition \ref{prop inverte}). The reason behind this choice is that, under the hypotheses of Theorem \ref{t:cohnotmin},  \eqref{e:mclosetoi*} implies
\begin{equation}\label{e:m=em}
\supp(\lambda)\cap [i_*,m] = \{m\},
\end{equation}
which simplifies part of the argument.  Note that \eqref{e:mclosetoi*} also implies $d(i_*,m)\le 2$, independently of the hypotheses of Theorem \ref{t:cohnotmin}. Moreover, under the hypotheses of Theorem \ref{t:cohnotmin}, $d(i_*,m)= 2$ only for $\lie g$ of type $E_6$ with $d(k,i_*)=1$, a case that can happen only under hypothesis (iii) of the theorem.

We have to show that, for all $\mu\in P^+$ such that $\mu<\lambda$, 
	\begin{equation}\label{e:D4mindef}
	\text{either}\ \ m_\mu(V_q(\bs\omega))\ge m_\mu(V_q(\bs\varpi)) \ \ \text{or}\ \ \exists\ \
	\mu'>\mu \ \ \text{s.t.}\ \ m_{\mu'}(V_q(\bs\varpi))< m_{\mu'}(V_q(\bs\omega)).
	\end{equation}
It obviously suffices to consider the case that $m_\mu(V_q(\bs\varpi))>0$. 	For $i\in\partial I$, let $j_i\in[i_*,i]$ be the element satisfying
\begin{equation*}
J_\mu\cap (j_i,i]=\emptyset \quad\text{and}\quad [i_*,j_i]\subseteq J_\mu
\end{equation*}
given by Lemma \ref{l:many0}(b). Lemma \ref{l:<nuk} implies 
\begin{equation*}
l_\lambda\in(i_*, j_l] \quad\text{for}\quad l\in\partial I_k.
\end{equation*}
If $k_\lambda\in [i_*,j_k]$, then $\mu\le\nu$ and, since $m_\mu(V_q(\bs\varpi))>0$, it follows from part (d) of Proposition \ref{p:outer} that $\mu<\nu$. Parts (c) and (d) of Proposition \ref{p:outer} imply that the second option in \eqref{e:D4mindef} is satisfied with $\mu'=\nu$. We claim that 
\begin{equation*}
k_\lambda\notin [i_*,j_k] \quad\Rightarrow\quad 	m_\mu(V_q(\bs\omega))= m_\mu(V_q(\bs\varpi)),
\end{equation*}
which completes the proof of \eqref{e:D4mindef} and, hence, of Theorem \ref{t:cohnotmin}. The claim clearly follows if we prove that, for $\bs{\pi}\in\{\bs{\omega},\bs{\varpi}\}$,  we have 
\begin{equation}\label{e:tpneeded}
k_\lambda\notin [i_*,j_k] \quad\Rightarrow\quad 	V_q(\bs{\pi}_{J_\mu})\cong  V_q\left(\bs\pi^{\{m\}}_{J_\mu}\right)\otimes V_q\left(\bs\pi^{[i_*,l]}_{J_\mu}\right),
\end{equation}
where $l$ is the unique element of $\partial I\setminus\{k,m\}$. Note that we have used \eqref{e:m=em} here.

Suppose first that $$j_k=i_*$$
which implies that $J_\mu$ is of type $A$. 
In order to prove \eqref{e:tpneeded}, we shall use Theorem \ref{t:tpa}. For $i\in I$, set $\lambda_i=\lambda(h_i)$.  Since $\bs{\omega}$ is coherent, there exist $a\in\mathbb F^\times$ and $\epsilon\in\{-1,1\}$ such that
\begin{equation*}
\bs{\omega} = \bs{\omega}_{k_\lambda,a,\lambda_{k_\lambda}}\ \bs{\omega}_{m,a_m,\lambda_m}\prod_{i\in (i_*,l]}\bs\omega_{i,a_i,\lambda_i}
\end{equation*}
with
\begin{equation}
a_m=aq^{\epsilon(\lambda_{k_\lambda}+\lambda_m+d(m,i_*)+1)} \qquad\text{and}\qquad  a_i=aq^{\epsilon\left(\lambda_{k_\lambda}+\lambda_i+d(i,i_*)+1+2\sum\limits_{j\in (i_*,i)}\lambda_j\right)}
\end{equation}
for all $i\in(i_*,l]$. Then, by definition of $\bs{\varpi}$, we have
\begin{equation*}
\bs{\varpi} = \bs{\omega}_{k_\lambda,a,\lambda_{k_\lambda}}\ \bs{\omega}_{m,b_m,\lambda_m}\prod_{i\in (i_*,l]}\bs\omega_{i,a_i,\lambda_i}
\end{equation*}
with
\begin{equation}
b_m=aq^{-\epsilon\left(\lambda_{k_\lambda}+\lambda_m+d(m,i_*)+1\right)}.
\end{equation}
Since the irreducibility of $V_q(\bs{\pi}_{J_\mu})$ is independent of the value of $\epsilon$, in order to avoid stating the version of Theorem \ref{t:tpa} for decreasing minimal affinizations (see Remark \ref{r:tpa}), we shall choose $\epsilon$ so that we place ourselves in the context of  Theorem \ref{t:tpa} as stated here. Thus, we identify $J_\mu$ with a diagram of type $A_n$ by letting $l$ be identified with $1$ and $m$ with $n$ and choose $\epsilon=-1$. With these choices, \eqref{e:tpneeded} follows if  the pair $\left(\bs\pi^{[i_*,l]}_{J_\mu},\bs\pi^{\{m\}}_{J_\mu}\right)$ does not satisfy any of the conditions of Theorem \ref{t:tpa} for $\bs{\pi}\in\{\bs{\omega},\bs{\varpi}\}$.  Note also that $l_\lambda$ corresponds to $j$ in the statement of Theorem \ref{t:tpa} and, hence, $a_{l_\lambda}$ corresponds to $a$ there and the ratio $q^s = b/a$ corresponds to $a_m/a_{l_\lambda}$, in the case $\bs{\pi}=\bs{\omega}$, and $b_m/a_{l_\lambda}$ for $\bs{\pi}=\bs{\varpi}$. This gives,
\begin{equation}\label{e:thess}
	\begin{aligned}
	   &\bs{\pi} = \bs{\omega} \qquad\Rightarrow\qquad && s = \lambda_{l_\lambda} -\lambda_m + d(l_\lambda,i_*) - d(m,i_*),\\ 
	   &\bs{\pi} = \bs{\varpi} \qquad\Rightarrow\qquad && s = 2\lambda_{k_\lambda}+\lambda_{l_\lambda}+\lambda_m + d(m,i_*) + d(l_\lambda,i_*) + 2.
	\end{aligned}
\end{equation}
Since $\lambda_m$ corresponds to $\eta$ in the statement of Theorem \ref{t:tpa}, the negation of its condition (i) is
\begin{equation}\label{e:condforirri}
\begin{aligned}
 s +&  \lambda_m + d(i,m) + 2 +\lambda_{l_\lambda} + d(i,l_\lambda) + 2\sum_{p\in [i,l_\lambda)}\lambda_p \ne 2t \\  
 &\text{for all}\quad i\in\supp(\lambda)\cap[l,i_*],\quad 1\le t\le \min\{\lambda_i,\lambda_m\},
\end{aligned}
\end{equation}
while that of condition (ii) is 
\begin{equation}\label{e:condforirrii}
s -  \lambda_m - \lambda_{l_\lambda} - d(l_\lambda,m) - 2 \ne  -2t \qquad\text{for all}\qquad  1\le t\le \min\{\, _l|\lambda|_{i_*},\lambda_m\}.
\end{equation}
Thus, to prove  \eqref{e:tpneeded} in the case $d(k_\lambda,i_*)=1$, it suffices to check that, for $s$ as in \eqref{e:thess}, \eqref{e:condforirri} and \eqref{e:condforirrii} hold.  

For $\bs\pi=\bs{\omega}$,  \eqref{e:condforirri} becomes
	\begin{equation*}
	2(\lambda_i-t)+2\lambda_{l_\lambda} + d(l_\lambda,i_*) + d(i,m) + d(i,l_\lambda) + (2-d(i_*,m)) + 2\sum_{p\in (i,l_\lambda)}\lambda_p \ne 0,
	\end{equation*}
	for all $1\le t\le \min\{\lambda_i,\lambda_m\}$, which is true since, for such $t$, all the summands above are nonnegative and several of them are not zero (e.g., $\lambda_{l_\lambda}\ne 0$).
	Equation \eqref{e:condforirrii} for $\bs{\pi=\bs{\omega}}$ becomes
	\begin{equation*}
	2(\lambda_m-t) + (d(l_\lambda,m)-d(l_\lambda,i_*)) + d(m,i_*) +2 \ne 0
	\end{equation*}
	for all $1\le t\le \min\{\, _l|\lambda|_{i_*},\lambda_m\}$. As before, we see that, for such $t$, all summands are nonnegative and several are positive, completing the proof of \eqref{e:tpneeded} in the case $d(k_\lambda,i_*)=1$ and $\bs{\pi}=\bs{\omega}$.	
	For $\bs\pi=\bs{\varpi}$, \eqref{e:condforirri} becomes 
	\begin{equation*}
	2(\lambda_{k_\lambda}+\lambda_{l_\lambda} + \lambda_i +\lambda_m -t) + d(l_\lambda,i_*) + d(i,m) + d(i,l_\lambda) +d(m,i_*)+4 + 2\sum_{p\in (i,l_\lambda)}\lambda_p \ne 0
	\end{equation*}
	for all $1\le t\le \min\{\lambda_i,\lambda_m\}$, which is easily seen to be true as before. On the other hand, equation \eqref{e:condforirrii} becomes
	\begin{equation*}
	2\lambda_{k_\lambda} +2t + (d(l_\lambda,i_*) - d(l_\lambda,m)) +d(m,i_*) \ne 0,
	\end{equation*}
	which clearly holds for all $t\ge 1$.

Finally, suppose $$j_k\ne i_*$$
which, together with the hypothesis in \eqref{e:tpneeded}, implies we must be under the hypothesis (i) or (iii) of Theorem \ref{t:cohnotmin}. In particular,   $J_\mu$ is of type $D_n$ with $n\ge 5$, $d(m,i_*)=1$, and $l_\lambda=l$.  Hence, we need to check that $\left(\bs\pi^{\{m\}}_{J_\mu},\bs\pi^{\{l\}}_{J_\mu}\right)$ does not satisfy any of the conditions of Corollary \ref{c:tpd}(a), in case hypothesis (i) is satisfied, and  of Corollary \ref{c:tpd}(b), in case hypothesis (iii) is satisfied, with $l$ corresponding to the non-spin node. This time there exist $a\in\mathbb F^\times$ and $\epsilon\in\{-1,1\}$ such that
\begin{equation*}
\bs{\omega} = \bs{\omega}_{k_\lambda,a,\lambda_{k_\lambda}}\ \bs{\omega}_{m,a_m,\lambda_m}\ \bs\omega_{l,a_l,\lambda_l} \qquad\text{and}\qquad \bs{\varpi} = \bs{\omega}_{k_\lambda,a,\lambda_{k_\lambda}}\ \bs{\omega}_{m,b_m,\lambda_m}\ \bs\omega_{l,a_l,\lambda_l}
\end{equation*}
with
\begin{equation}
 \frac{a_m}{a}=q^{\epsilon(\lambda_{k_\lambda}+\lambda_m+d(m,k_\lambda))}=\frac{a}{b_m} \qquad\text{and}\qquad  \frac{a_l}{a}=q^{\epsilon\left(\lambda_{k_\lambda}+\lambda_l+d(m,k_\lambda)+t\right)},
\end{equation}
where $t=0$ for hypothesis (i) (both $l$ and $m$ are spin nodes) and $t=1$ for hypothesis (iii) ($l$ and $k$ are the extremal nodes of the subdiagram of type $A_5$). For checking \eqref{e:tpneeded} with $\bs{\pi}=\bs{\omega}$, in case of (i), we need to check that
\begin{equation*}
  \lambda_m-\lambda_l \ne \pm(\lambda_m+\lambda_l+2(2s-p))  \qquad\text{for all}\qquad 1\le p\le\min\{\lambda_m,\lambda_l\}, \ 1\le s\le \lfloor (\#J_\mu-1)/2\rfloor.
\end{equation*}
But equality holds if and only if there exist $i\in\{m,l\}$ and $s,p$ in the above ranges such that
\begin{equation*}
  \lambda_i-p+2s =0,
\end{equation*} 
which is impossible since $\lambda_i-p\ge 0$ and $s>0$. Similarly, in case of (iii), one easily checks that
\begin{equation*}
 \lambda_m-\lambda_l-1 \ne \pm(\lambda_m+\lambda_l+\#J_\mu-2p)  \qquad\text{for all}\qquad 1\le p\le\min\{\lambda_m,\lambda_l\}.
\end{equation*}
For checking \eqref{e:tpneeded} with $\bs{\pi}=\bs{\varpi}$, in case of (i), we need to check that
\begin{gather*}
2\lambda_{k_\lambda}+\lambda_m+\lambda_l +2d(k_\lambda,m)\ne \pm(\lambda_m+\lambda_l+2(2s-p))\\  \qquad\text{for all}\qquad\\ 1\le p\le\min\{\lambda_m,\lambda_l\}, \ 1\le s\le \lfloor (\#J_\mu-1)/2\rfloor.
\end{gather*}
But equality is impossible because
\begin{gather*}
\lambda_{k_\lambda} + d(k_\lambda,m)\ge d(j_k,m) +2 = \#J_\mu  \qquad\text{while}\qquad 2s-p\le \#J_\mu-1\\ \text{and}\\
\lambda_{k_\lambda} + \lambda_m+\lambda_l+ d(k_\lambda,m)> \#J_\mu +2 + 2\min\{\lambda_m,\lambda_l\} \qquad\text{while}\qquad p-2s\le \min\{\lambda_m,\lambda_l\}.
\end{gather*} 
In case of (iii), we have $\#J_\mu=5$, $d(k_\lambda,m)=d(k,m)=3$ and one easily checks that
\begin{equation*}
2\lambda_{k_\lambda}+\lambda_m+\lambda_l +2d(k_\lambda,m)+1 \ne \pm(\lambda_m+\lambda_l+\#J_\mu-2p)  \quad\text{for all}\quad 1\le p\le\min\{\lambda_m,\lambda_l\}.
\end{equation*}
This completes the proof  of Theorem \ref{t:cohnotmin}.

\bibliographystyle{amsplain}

\begin{thebibliography}{10}
\bibitem{cha:minr2}
V.~Chari, {\em Minimal affinizations of representations of quantum groups: the rank-2 case}, Publ. Res. Inst. Math. Sci. {\bf 31} (1995), 873--911, \url{https://doi.org/10.1007/BF00750760}.

\bibitem{cha:fer}
V. Chari, {\it On the fermionic formula and the Kirillov-Reshetikhin conjecture}, Int. Math. Res. Not. 2001 (2001), 629--654.

\bibitem{cha:braid}
V.~Chari, {\em Braid group actions and tensor products}, Int. Math. Res. Notices (2002), 357--382. \url{https://doi.org/10.1155/S107379280210612X}

\bibitem{cfk:cat}  
V. Chari, G. Fourier, T. Kahndai, {\em A categorical approach to Weyl modules}, Transf. Groups {\bf 15} (2010), 517--549.	

\bibitem{chhe:bey}  
V. Chari and D. Hernandez, {\em Beyond Kirillov-Reshetikhin modules}, Contemp. Math. {\bf 506} (2010), 49--81.


\bibitem{cm:qblock}
V.~Chari and A.~Moura, {\em Characters and blocks for finite-dimensional representations of quantum affine algebras}, Int. Math. Res. Notices (2005), no. 5, {257--298}. \url{https://doi.org/10.1155/IMRN.2005.257}

\bibitem{cm:kr}
V. Chari and A. Moura, {\em The restricted Kirillov-Reshetikhin modules for the	current and twisted current algebras}, Comm. Math. Phys. {\bf 266}
(2006), 431--454.

\bibitem{cp:small}
V. Chari and A. Pressley, {\em Small representations of quantum affine algebras}, Lett. Math. Phys. {\bf 30}
(1994), 131--145.

\bibitem{cp:book}
V. Chari and A. Pressely,  A guide to quantum groups, Cambridge University Press (1994).

\bibitem{cp:minsl}
V. Chari and A. Pressely, {\em Minimal affinizations of representations of quantum
	groups: the simply laced case}, J. of Algebra {\bf 184} (1996),
1--30.

\bibitem{cp:minirr}
V. Chari and A. Pressely, {\em Minimal affinizations of representations of quantum groups: the irregular case}, Lett. Math. Phys. {\bf 36} (1996), 247--266.

\bibitem{cp:weyl}
V. Chari and A. Pressely,  {\em Weyl modules for classical and quantum affine algebras}, Represent. Theory {\bf 5} (2001), 191--223. \url{https://doi.org/10.1090/S1088-4165-01-00115-7}

\bibitem{chi}
B. Duan, J.-R. Li, Y.-F. Luo, and Q.-Q. Zhang, {\em M-systems and cluster algebras}, Int. Math. Res. Not. {\bf 2016} (2016),  4449--4486, \url{https://doi.org/10.1093/imrn/rnv287}.

\bibitem{fremuk:qchar}
E. Frenkel and E. Mukhin, {\it Combinatorics of $q$-characters of
	finite-dimensional representations of quantum affine algebras},
Comm. Math. Phys. {\bf 216} (2001), 23--57. \url{https://doi.org/10.1007/s002200000323}

\bibitem{freres:qchar}
E. Frenkel and N. Reshetikhin, {\it The $q$-characters of representations of quantum affine algebras and deformations of $\mathcal{W}$-algebras}, Recent developments in quantum affine algebras and related topics (Raleigh, NC, 1998), Contemp. Math. {\bf	248} (1999), 163--205.

\bibitem{jap:rem}
G. Hatayama, A. Kuniba, M. Okado, T. Takagi, and Y. Yamada, {\em Remarks on the Fermionic Formula}. Contemp. Math. {\bf 248} (1999), 243--291.

\bibitem{her:min}
D.~Hernandez, {\em On minimal affinizations of representations of quantum groups}, Comm. Math. Phys. {\bf 277} (2007), 221--259, \url{https://doi.org/10.1007/s00220-007-0332-1}.

\bibitem{her:KRconj}
D.~Hernandez, {\em Kirillov-Reshetikhin Conjecture: the general case}, Int. Math. Res. Notices {\bf 2010} (2010), 149--193, \url{https://doi.org/10.1093/imrn/rnp121}.

\bibitem{hele:KRcluster}
D.~Hernandez and B. Leclerc, {\em A cluster algebra approach to q-characters of Kirillov-Reshetikhin modules}, J. European Math. Soc. {\bf 18} (2016), 1113--1159, \url{ http://dx.doi.org/10.4171/JEMS/609}.

\bibitem{jim:qan}
M. Jimbo, {\it A q-analogue of $U(gl(n + 1))$, Hecke algebra and the Yang-Baxter equation}, Lett. Math. Phys. {\bf 11} (1986), 257--252.


\bibitem{lina}
J.-R. Li and K. Naoi, {\em Graded limits of minimal affinizations over the quantum loop algebra of type $G_2$}, Algebr Represent Theor {\bf 19} (2016), 957--973, \url{https://doi.org/10.1007/s10468-016-9606-7}.

\bibitem{liqi}
J.-R. Li and L. Qiao, {\em Three-term recurrence relations of minimal affinizations of type $G_2$}, Journal of Lie Theory {\bf 27} (2017), 1119--1140.


\bibitem{lus:book}
G. Lusztig, Introduction to quantum groups, Birk\"auser (1993).

\bibitem{mou}
A.~Moura, {\em Restricted limits of minimal affinizations}, Pacific J. Math. {\bf 244} (2010), 359--397. \url{http://dx.doi.org/10.2140/pjm.2010.244.359}

\bibitem{mope:mine6}
A. Moura and F. Pereira, {\em Graded limits of minimal affinizations and beyond: the multiplicity free case for type $E_6$}, Algebra and Discrete Mathematics {\bf 12} (2011), 69--115.


\bibitem{hmp:tpa}
A. Moura and F. Pereira, {\em On tensor products of a minimal affinization with an extreme Kirillov-Reshetikhin module for type $A$}, to appear in Transformations Groups, \url{https://doi.org/10.1007/s00031-017-9462-5}.

\bibitem{hmp:tpd}
F. Pereira,  {\em On qcharacters and tensor products of spin Kirillov-Rehetihin modules},  in preparation.

\bibitem{naoi:grlim}
K. Naoi, {\em Demazure modules and graded limits of minimal	affinizations}, Rep. Theory {\bf 17} (2013), 524--556. \url{http://dx.doi.org/10.1090/S1088-4165-2013-00442-9}.

\bibitem{naoi:D}
K. Naoi, {\em Graded limits of minimal affinizations in Type $D$}, SIGMA {\bf 10} (2014), 047, 20 pages. \url{http://dx.doi.org/10.3842/SIGMA.2014.047}

\bibitem{fer:thesis}
F. Pereira, Classification of the type $D$ irregular minimal affinizations, Ph.D. Thesis, UNICAMP (2014), \url{http://repositorio.unicamp.br/jspui/handle/REPOSIP/307009}.

\end{thebibliography}

\end{document}